\documentclass{amsart}

\usepackage[T1]{fontenc}
\usepackage[utf8]{inputenc}
\usepackage{amsmath,amsthm,amsfonts,amscd,amssymb,eucal,latexsym,mathrsfs}
\usepackage{stmaryrd}
\usepackage{enumerate}
\usepackage{hyperref}
\usepackage[all]{xy}
\usepackage{etoolbox}
\usepackage{amssymb}
\usepackage{amsfonts}
\usepackage{graphicx}
\usepackage{mathtools}

\usepackage{tikz,tikz-cd}
\usetikzlibrary{shapes.geometric}
\usetikzlibrary{arrows}
\usetikzlibrary{calc}

\usetikzlibrary{positioning}
\usepackage{float}
\usepackage{MnSymbol}
\usetikzlibrary{matrix}
\usepackage{tkz-euclide}

\theoremstyle{plain}
\newtheorem{theorem}{Theorem}[section]
\newtheorem*{theorem*}{Theorem}

\newtheorem*{corollary*}{Corollary}

\newtheorem{lemma}[theorem]{Lemma}

\theoremstyle{definition}
\newtheorem{remark}[theorem]{Remark}

\newtheorem{definition}[theorem]{Definition}
\newtheorem*{definition*}{Definition}

\usetikzlibrary{calc,graphs}
\usepackage{xcolor}
\usetikzlibrary{arrows,decorations.pathmorphing}

\newcommand{\e}{\varepsilon}

\newcommand{\ZI}{\mathbb{Z}}

\DeclareMathOperator{\dom}{\mathrm{dom}}

\DeclareMathOperator{\inj}{\hookrightarrow}

\newcommand{\rk}{\mathrm{rk}}
\newcommand{\val}{\mathrm{val}}

\newcommand{\ts}{\textsection}

\newcommand{\up}[3]{\draw[xshift=#1cm,yshift=#2*.866cm,line width = 0.1mm] (0:0) -- (0:1) -- (60:1) -- (0:0) node at (.5,.3) {#3};}
\newcommand{\down}[3]{\draw[xshift=#1cm,yshift=#2*.866cm,line width = 0.1mm] (0:0) -- (60:1) -- (120:1) -- (0:0)  node at (0,.5) {#3};}
\newcommand{\upthick}[3]{\draw[xshift=#1cm,yshift=#2*.866cm,line width = 0.4mm] (0:0) -- (0:1) -- (60:1) -- (0:0) node at (.5,.3) {#3};}
\newcommand{\downthick}[3]{\draw[xshift=#1cm,yshift=#2*.866cm,line width = 0.4mm] (0:0) -- (60:1) -- (120:1) -- (0:0)  node at (0,.5) {#3};}

\newcommand{\uplarge}[6]{
	\up{#1}{#2}{#3}  
	
	\pgfmathparse{#1 + 1}
    \let \xshift \pgfmathresult
    \up{\xshift}{#2}{#5}
    \down{\xshift}{#2}{#4}
    
	\pgfmathparse{#1 + .5}
    \let \xshift \pgfmathresult
    \pgfmathparse{#2 + 1}
    \let \yshift \pgfmathresult
    \up{\xshift}{\yshift}{#6}
 
  }

\newcommand{\downlarge}[6]{
	\down{#1}{#2}{#3}  
	\up{#1}{#2}{#4}

	\pgfmathparse{#1 + 1}
    \let \xshift \pgfmathresult
    \down{\xshift}{#2}{#5}

	\pgfmathparse{#1 + .5}
    \let \xshift \pgfmathresult
    \pgfmathparse{#2 - 1}
    \let \yshift \pgfmathresult
    \down{\xshift}{\yshift}{#6}
 
  }

\newcommand{\hexagon}[8]{
    
    \up{#1}{#2}{#3}
    \down{#1}{#2}{#4}
     
    \pgfmathparse{#1 - 1}
    \let \xshift \pgfmathresult
    \up{\xshift}{#2}{#5}
    
    \pgfmathparse{#1 - 0.5}
    \let \xshift \pgfmathresult
    \pgfmathparse{#2 - 1}
    \let \yshift \pgfmathresult
    \down{\xshift}{\yshift}{#6}
    \up{\xshift}{\yshift}{#7}
     
    \pgfmathparse{#1 + 0.5}
    \let \xshift \pgfmathresult
    \down{\xshift}{\yshift}{#8}
           
}

\newcommand{\hexagonsmall}[8]{
    
    \draw[xshift=#1cm,yshift=#2*.866cm,line width = 0.1mm] (0:0) -- (0:1) -- (60:1) -- (0:0) node at (.3,.15) {\tiny #3};

\draw[xshift=#1cm,yshift=#2*.866cm,line width = 0.1mm] (0:0) -- (60:1) -- (120:1) -- (0:0)  node at (0,.25) {\tiny #4};

    \pgfmathparse{#1 - 1}
    \let \xshift \pgfmathresult
    \draw[xshift=\xshift cm,yshift=#2*.866cm,line width = 0.1mm] (0:0) -- (0:1) -- (60:1) -- (0:0) node at (.7,.15) {\tiny #5};
    
    \pgfmathparse{#1 - 0.5}
    \let \xshift \pgfmathresult
    \pgfmathparse{#2 - 1}
    \let \yshift \pgfmathresult

 \draw[xshift=\xshift cm,yshift=\yshift*.866cm,line width = 0.1mm] (0:0) -- (60:1) -- (120:1) -- (0:0)  node at (0.2,.7) {\tiny #6};

    \draw[xshift=\xshift cm,yshift=\yshift*.866cm,line width = 0.1mm] (0:0) -- (0:1) -- (60:1) -- (0:0) node at (.5,.6) {\tiny #7};
     
    \pgfmathparse{#1 + 0.5}
    \let \xshift \pgfmathresult
   \draw[xshift=\xshift cm,yshift=\yshift*.866cm,line width = 0.1mm] (0:0) -- (60:1) -- (120:1) -- (0:0)  node at (-.2,.7) {\tiny #8};

}

\newcommand{\hexagonthick}[8]{
    
    \up{#1}{#2}{#3}
    \down{#1}{#2}{#4}
     
    \pgfmathparse{#1 - 1}
    \let \xshift \pgfmathresult
    \up{\xshift}{#2}{#5}
    
    \pgfmathparse{#1 - 0.5}
    \let \xshift \pgfmathresult
    \pgfmathparse{#2 - 1}
    \let \yshift \pgfmathresult
    \down{\xshift}{\yshift}{#6}
    \up{\xshift}{\yshift}{#7}
     
    \pgfmathparse{#1 + 0.5}
    \let \xshift \pgfmathresult
    \down{\xshift}{\yshift}{#8}
        
    \draw[xshift=#1cm,yshift=#2*0.866cm,line width=0.25mm] node[circle, fill=olive, draw=white, scale=0.6] at (0,0) {};
   
    \draw[xshift=#1cm,yshift=#2*0.866cm,line width=0.25mm] (0:1) -- (60:1) -- (120:1) -- (180:1) -- (240:1) -- (300:1) -- cycle;
}

\newcommand{\figureuvw}[3]{

\up{0}{0}{0}
\down{0.5}{-1}{#1}
\down{1}{0}{#2}
\down{0}{0}{#3}

}

\title{The odd triangle ring puzzle problem}
\author{Sylvain Barré, Othmane Oukrid, Mika\"el Pichot}

\begin{document}

\maketitle

\begin{abstract} Ring puzzles are tessellations of the Euclidean plane respecting local constraints around vertices. Such puzzles may arise in geometric group theory, for example,  as embedded flat planes in certain CAT(0) complexes of dimension 2.  In the present paper, we solve the odd ring puzzle problem, which is associated with the unique odd Moebius--Kantor CAT(0) complex  by the method of Sidon sequences. We prove that there are precisely three families of such puzzles, two uncountable families, and a finite family of twelve exceptional puzzles.    
\end{abstract}

\section{Introduction}

  Ring puzzles are tiling puzzles in which a set of flat shapes are  assembled together to form a tessellation of the Euclidean plane, with local constraints at vertices which restricts the permitted neighbourhoods. These constraints are determined by a second set of ring-shaped pieces  placed around the vertices, and these two sets, of flat shapes and rings,  define the puzzle problem entirely. In practice, these sets can often be described by marked polygons and, respectively, marked circles of radius one. 
    When all shapes are equilateral triangles, we say that the ring puzzle problem is a triangle ring puzzle problem. In this case, the rings are marked circles equally subdivided into six arcs. The goal of the present paper is to solve the following triangle ring puzzle problem.
     
 Let $S=\ZI/3\ZI$.  Both the shape set and the ring set are marked by elements in $S$.  The shape set contains  three equilateral triangles endowed with the following marking:

\begin{figure}[H]
 \begin{tikzpicture}
 \up{0}{0}{$0$}
     \draw node at (.5,-.2) {\tiny 0};
     \draw node at (.9,.5) {\tiny 1};
     \draw node at (0.1,.5) {\tiny 2};
    \end{tikzpicture}\ \ \ \ \ \ \ \ \ \ \ \
 \begin{tikzpicture}
 \up{0}{0}{$1$}
     \draw node at (.5,-.2) {\tiny 0};
     \draw node at (.9,.5) {\tiny 1};
     \draw node at (0.1,.5) {\tiny 2};
    \end{tikzpicture}\ \ \ \ \ \ \ \ \ \ \ \ 
 \begin{tikzpicture}
 \up{0}{0}{$2$}
     \draw node at (.5,-.2) {\tiny 0};
     \draw node at (.9,.5) {\tiny 1};
     \draw node at (0.1,.5) {\tiny 2};
    \end{tikzpicture}
    \caption{The shape set of the odd ring puzzle problem}\label{Fig 1}
     \end{figure}
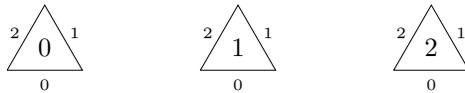
 \noindent Here, both the interior of each triangle, and its sides, are marked by elements in $S$.  The ring set is of the form $\Theta:=\bigsqcup_{s\in S} \Theta_s$, where $\Theta_s$ is a set of three rings marked by the elements of $S$ as follows:

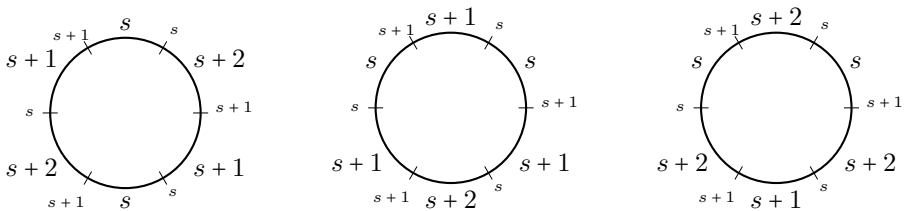
\begin{figure}[H]
 \begin{tikzpicture}
    \tikzset{narrow/.style={inner sep=3pt}}
    \draw [thick] (0,0) circle[radius=1cm];
    \foreach \ang/\lab [evaluate=\ang as \anc using \ang+180]
        in {30/$s+2$, 90/$s$, 150/$s+1$, -150/$s+2$, -90/$s$,-30/$s+1$}
        \draw [ultra thin] (\ang+30:.9)--(\ang+30:1.1) node{} (\ang:1) node[narrow, anchor=\anc]{\lab};
        \foreach \ang/\lab [evaluate=\ang as \anc using \ang+180]
        in {30/{$s$}, 95/{$s+1$}, 150/{$s$}, -150/{$s+1$}, -90/{$s$},-23/{$s+1$}}
        \draw [ultra thin, anchor=\anc] (\ang+35:1.1) node{\tiny\lab};
    \end{tikzpicture}\ \ \ \ \ \ \
 \begin{tikzpicture}
    \tikzset{narrow/.style={inner sep=3pt}}
    \draw [thick] (0,0) circle[radius=1cm];
    \foreach \ang/\lab [evaluate=\ang as \anc using \ang+180]
        in {30/$s$, 90/$s+1$, 150/$s$, -150/$s+1$, -90/$s+2$,-30/$s+1$}
        \draw [ultra thin] (\ang+30:.9)--(\ang+30:1.1) node{} (\ang:1) node[narrow, anchor=\anc]{\lab};
        \foreach \ang/\lab [evaluate=\ang as \anc using \ang+180]
        in {30/{$s$}, 95/{$s+1$}, 150/{$s$}, -150/{$s+1$}, -90/{$s$},-23/{$s+1$}}
        \draw [ultra thin, anchor=\anc] (\ang+35:1.1) node{\tiny\lab};
    \end{tikzpicture}\ \ \ \ \ \ \
 \begin{tikzpicture}
    \tikzset{narrow/.style={inner sep=3pt}}
    \draw [thick] (0,0) circle[radius=1cm];
    \foreach \ang/\lab [evaluate=\ang as \anc using \ang+180]
        in  {30/$s$, 90/$s+2$, 150/$s$, -150/$s+2$, -90/$s+1$,-30/$s+2$}
        \draw [ultra thin] (\ang+30:.9)--(\ang+30:1.1) node{} (\ang:1) node[narrow, anchor=\anc]{\lab};
        \foreach \ang/\lab [evaluate=\ang as \anc using \ang+180]
        in {30/{$s$}, 95/{$s+1$}, 150/{$s$}, -150/{$s+1$}, -90/{$s$},-23/{$s+1$}}
        \draw [ultra thin, anchor=\anc] (\ang+35:1.1) node{\tiny\lab};
    \end{tikzpicture}
    \caption{The set $\Theta_s$ of rings ($s\in S$) in the odd ring puzzle problem}\label{Fig 2} 
\end{figure}

\noindent Again, both the edges and the vertices in every ring are marked by elements in $S$.  A solution to this ring puzzle problem, called an odd ring puzzle, is a tessellation of the Euclidean plane by equilateral triangles having markings on every (interior) triangle, and every edges (sides), and which respects the given local constraints in the obvious way. Formally, we  require that, for every vertex $x$ in the tessellation, there exists an $R\in \Theta$, and a label preserving isomorphism $\rho\colon L_x \to R$, from the link $L_x$ of $x$ (i.e., any sphere of sufficiently small radius around $x$, endowed with the angular metric), where every vertex and, respectively, every arc in $L_x$ is labeled by the corresponding elements in $S$ in the adjacent edge and, respectively, triangle, of the marked tessellation.

 Ring puzzles were introduced in \cite{autf2} in order to study the geometric properties of the automorphism group  of the free group on two generators. The puzzle problem described above  was introduced more recently in \cite[\ts 7.2]{Sidon} and belongs to a large family of such problems obtained by ``the method of Sidon sequences''. It is associated with the group of automorphisms of the unique odd Moebius--Kantor complex (we shall provide more details on this complex in \ts\ref{S - Roots}) and we shall refer to it as the odd ring puzzle problem for simplicity.   Theorem \ref{T - main theorem} below is the first non trivial case of a complete solution of a Sidon type ring puzzle problem.

The following is a simple but important property of the odd ring puzzle problem.

\begin{lemma}\label{L - edge labelling} The ring set in Fig.\ \ref{Fig 2} defines a unique edge marking (by elements in $S$) in every odd ring puzzle.
\end{lemma}

\begin{proof} The unique edge marking is

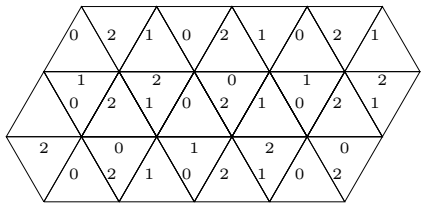
\begin{figure}[H]
 \begin{tikzpicture}
 \hexagon{0}{0}{}{}{}{}{}{};
 \hexagon{1}{0}{}{}{}{}{}{};
 \hexagon{2}{0}{}{}{}{}{}{};
 \hexagon{3}{0}{}{}{}{}{}{};
 \hexagon{.5}{1}{}{}{}{}{}{};
 \hexagon{1.5}{1}{}{}{}{}{}{};
 \hexagon{2.5}{1}{}{}{}{}{}{};
 \hexagon{3.5}{1}{}{}{}{}{}{};
     \draw node at (-.5,-.15) {\tiny 2};
     \draw node at (.5,-.15) {\tiny 0};
     \draw node at (1.5,-.15) {\tiny 1};
     \draw node at (2.5,-.15) {\tiny 2};
     \draw node at (3.5,-.15) {\tiny 0};
     \draw node at (-.1,.45) {\tiny 0};
     \draw node at (.9,.45) {\tiny 1};
     \draw node at (1.9,.45) {\tiny 2};
     \draw node at (2.9,.45) {\tiny 0};
     \draw node at (3.9,.45) {\tiny 1};
     \draw node at (0.4,.45) {\tiny 2};
     \draw node at (1.4,.45) {\tiny 0};
     \draw node at (2.4,.45) {\tiny 1};
     \draw node at (3.4,.45) {\tiny 2};
     
     \draw node at (-.1,-.5) {\tiny 0};
       \draw node at (.9,-.5) {\tiny 1};
     \draw node at (1.9,-.5) {\tiny 2};
     \draw node at (2.9,-.5) {\tiny 0};
     \draw node at (0.4,-.5) {\tiny 2};
     \draw node at (1.4,-.5) {\tiny 0};
     \draw node at (2.4,-.5) {\tiny 1};
     \draw node at (3.4,-.5) {\tiny 2};
     
          \draw node at (0,.75) {\tiny 1};
     \draw node at (1,.75) {\tiny 2};
     \draw node at (2,.75) {\tiny 0};
     \draw node at (3,.75) {\tiny 1};
     \draw node at (4,.75) {\tiny 2};
     \begin{scope}[shift={(.5,.9)}]
          \draw node at (-.1,.45) {\tiny 2};
     \draw node at (.9,.45) {\tiny 0};
     \draw node at (1.9,.45) {\tiny 1};
     \draw node at (2.9,.45) {\tiny 2};
     \draw node at (-.6,.45) {\tiny 0};
     \draw node at (0.4,.45) {\tiny 1};
     \draw node at (1.4,.45) {\tiny 2};
     \draw node at (2.4,.45) {\tiny 0};
     \draw node at (3.4,.45) {\tiny 1};
\end{scope}

    \end{tikzpicture}
     \end{figure}
     \noindent It is unique in the sense that it is determined from the choice of a marking on any initial triangle. Note that the edge marking on every triangle is necessarily injective.  
          \end{proof}
  
It follows by Lemma \ref{L - edge labelling}  that the $S$-marking of the edge set in an odd ring puzzle (and in fact, in every Sidon sequence ring puzzle as described in \cite{Sidon}) is essentially optional: once an (injective) edge marking is chosen for the initial triangle, the  marking is uniquely determined on the whole tessellation. Each edge is labelled by an element in $S$, in such a way  that every triangle in the tessellation contains all three edge labels in $S$---while every vertex is incident to only two.   
 It is therefore superfluous to provide markings for the edges in an odd puzzle, once the initial triangle is marked as a reference point. 
  
  Lemma \ref{L - edge labelling} also shows that the rings in the set $\Theta_s$ can exclusively be used at vertices which are adjacent to edges with labels $s$ and $s+1$.  Furthermore, since the label $s=0$ belongs to an arc in every ring, it is obvious that every solution to the odd ring puzzle problem must contain the triangle marked with $s=0$. 

Therefore, a natural approach to solve the odd ring puzzle problem would begin with a triangle with a zero marking in its interior (and a fixed injective marking of its three sides),  and start to  ``build around it'' using the permitted neighbourhoods at vertices. Unfortunately, this does not to work (as we shall explain in \ts \ref{S - Meso}). 
The goal of this paper is to present an alternative approach which relies on ideas from geometric group theory, and leads to a complete classification. A basic concept  which proves particularly useful for this specific puzzle is that of root distribution. In the present context, a root distribution on the equilateral triangle lattice simply amounts  to choosing, for every vertex in the lattice, precisely one of the three simplicial directions at this vertex. 

The following is our main result. The first step in the proof is to classify all root distributions which are associated to our puzzle problem.  We refer to \ts\ref{S - Roots} for  precise definitions and  a classification.

\begin{theorem}[Classifying the odd ring puzzles] \label{T - main theorem}
Let $S=\ZI/3\ZI$ and consider the shape and ring sets defined in Fig.\ \ref{Fig 1} and Fig.\ \ref{Fig 2}, respectively. There are precisely three families of solutions:
\begin{enumerate}
\item Uncountably many puzzles obtained as unions of flat strips of height 1 and period 6;
\item Uncountably many puzzles obtained as unions of flat strips of height 2 and period 6; 
\item Twelve exceptional puzzles not belonging to Case (1) or Case (2). 
\end{enumerate}
\end{theorem}

Furthermore, the flat strips that can be used in Case (1) and (2) are restricted, and the twelves puzzles in Case (3) can be described entirely.

In case (1), only four different flat strips can be used to form an odd ring puzzle. They are of period 6, as depicted with their markings in $\ZI/3\ZI$ in the figures below:

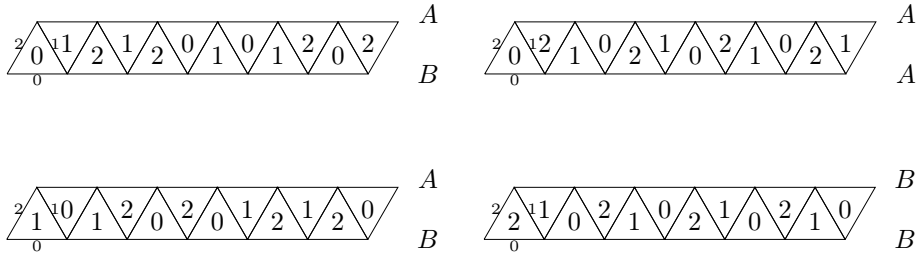
\begin{figure}[H]
 \begin{tikzpicture}[scale=.8]
 
     \draw node at (.5,-.1) {\tiny 0};
     \draw node at (.8,.5) {\tiny 1};
     \draw node at (0.2,.5) {\tiny 2};
     
\up{0}{0}{0}
\down{1}{0}{1}
\up{1}{0}{2}
\down{2}{0}{1}
\up{2}{0}{2}
\down{3}{0}{0}
\up{3}{0}{1}
\down{4}{0}{0}
\up{4}{0}{1}
\down{5}{0}{2}
\up{5}{0}{0}
\down{6}{0}{2}

  \draw node at (7,1) {$A$};
 \draw node at (7,0) {$B$};

    \end{tikzpicture}\ \ \ \ 
 \begin{tikzpicture}[scale=.8]
 
     \draw node at (.5,-.1) {\tiny 0};
     \draw node at (.8,.5) {\tiny 1};
     \draw node at (0.2,.5) {\tiny 2};
     
\up{0}{0}{0}
\down{1}{0}{2}
\up{1}{0}{1}
\down{2}{0}{0}
\up{2}{0}{2}
\down{3}{0}{1}
\up{3}{0}{0}
\down{4}{0}{2}
\up{4}{0}{1}
\down{5}{0}{0}
\up{5}{0}{2}
\down{6}{0}{1}

  \draw node at (7,1) {$A$};
 \draw node at (7,0) {$A$};

    \end{tikzpicture}
    \end{figure}

\begin{figure}[H]
 \begin{tikzpicture}[scale=.8]
 
     \draw node at (.5,-.1) {\tiny 0};
     \draw node at (.8,.5) {\tiny 1};
     \draw node at (0.2,.5) {\tiny 2};
     
\up{0}{0}{1}
\down{1}{0}{0}
\up{1}{0}{1}
\down{2}{0}{2}
\up{2}{0}{0}
\down{3}{0}{2}
\up{3}{0}{0}
\down{4}{0}{1}
\up{4}{0}{2}
\down{5}{0}{1}
\up{5}{0}{2}
\down{6}{0}{0}

  \draw node at (7,1) {$A$};
 \draw node at (7,0) {$B$};
    \end{tikzpicture}
\ \ \ 
 \begin{tikzpicture}[scale=.8]
 
     \draw node at (.5,-.1) {\tiny 0};
     \draw node at (.8,.5) {\tiny 1};
     \draw node at (0.2,.5) {\tiny 2};
     
\up{0}{0}{2}
\down{1}{0}{1}
\up{1}{0}{0}
\down{2}{0}{2}
\up{2}{0}{1}
\down{3}{0}{0}
\up{3}{0}{2}
\down{4}{0}{1}
\up{4}{0}{0}
\down{5}{0}{2}
\up{5}{0}{1}
\down{6}{0}{0}

  \draw node at (7,1) {$B$};
 \draw node at (7,0) {$B$};

    \end{tikzpicture}

    \end{figure}

 In the figures we have indicated by the letters $A$ and $B$ how to stack the strips  together to form a puzzle. We observe that strips 2 and 4 (i.e., on the right in the figure) are in fact  glided symmetry of themselves. Therefore, there are uncountably many odd puzzles,  which are parametrized by a transversal tree of order 3.

In case (2), only three different strips  can be used to form an ring puzzle in this category. They also are strips of period six as shown in the three figures below:

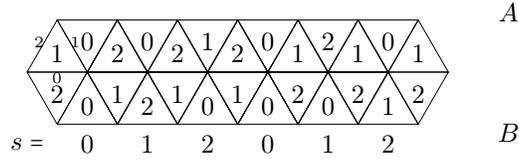
\begin{figure}[H]
 \begin{tikzpicture}[scale=.8]
 
     \draw node at (.5,-.1) {\tiny 0};
     \draw node at (.8,.5) {\tiny 1};
     \draw node at (0.2,.5) {\tiny 2};
     
\hexagon{1}{0}{2}{0}{1}{2}{0}{1}
\hexagon{2}{0}{2}{0}{}{}{2}{1}
\hexagon{3}{0}{2}{1}{}{}{0}{1}
\hexagon{4}{0}{1}{0}{}{}{0}{2}
\hexagon{5}{0}{1}{2}{}{}{0}{2}
\hexagon{6}{0}{1}{0}{}{}{1}{2}
 
     \draw node at (0,-1.2) {$s=$};
     \draw node at (1,-1.2) {0};
     \draw node at (2,-1.2) {1};
     \draw node at (3,-1.2) {2};
     \draw node at (4,-1.2) {0};
     \draw node at (5,-1.2) {1};
     \draw node at (6,-1.2) {2};
 
 \draw node at (8,1) {$A$};
 \draw node at (8,-1) {$B$};
 
    \end{tikzpicture}
    \end{figure}

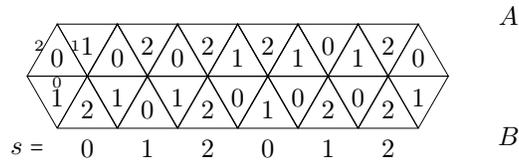
\begin{figure}[H]
 \begin{tikzpicture}[scale=.8]
 
     \draw node at (.5,-.1) {\tiny 0};
     \draw node at (.8,.5) {\tiny 1};
     \draw node at (0.2,.5) {\tiny 2};
     
\hexagon{1}{0}{0}{1}{0}{1}{2}{1}
\hexagon{2}{0}{0}{2}{}{}{0}{1}
\hexagon{3}{0}{1}{2}{}{}{2}{0}
\hexagon{4}{0}{1}{2}{}{}{1}{0}
\hexagon{5}{0}{1}{0}{}{}{2}{0}
\hexagon{6}{0}{0}{2}{}{}{2}{1}
 
     \draw node at (0,-1.2) {$s=$};
     \draw node at (1,-1.2) {0};
     \draw node at (2,-1.2) {1};
     \draw node at (3,-1.2) {2};
     \draw node at (4,-1.2) {0};
     \draw node at (5,-1.2) {1};
     \draw node at (6,-1.2) {2};

  \draw node at (8,1) {$A$};
 \draw node at (8,-1) {$B$};

    \end{tikzpicture}
    \end{figure}

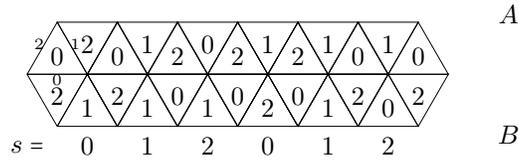
\begin{figure}[H]
 \begin{tikzpicture}[scale=.8]
 
     \draw node at (.5,-.1) {\tiny 0};
     \draw node at (.8,.5) {\tiny 1};
     \draw node at (0.2,.5) {\tiny 2};
     
\hexagon{1}{0}{0}{2}{0}{2}{1}{2}
\hexagon{2}{0}{2}{1}{}{}{1}{0}
\hexagon{3}{0}{2}{0}{}{}{1}{0}
\hexagon{4}{0}{2}{1}{}{}{2}{0}
\hexagon{5}{0}{0}{1}{}{}{1}{2}
\hexagon{6}{0}{0}{1}{}{}{0}{2}
 
     \draw node at (0,-1.2) {$s=$};
     \draw node at (1,-1.2) {0};
     \draw node at (2,-1.2) {1};
     \draw node at (3,-1.2) {2};
     \draw node at (4,-1.2) {0};
     \draw node at (5,-1.2) {1};
     \draw node at (6,-1.2) {2};
 
  \draw node at (8,1) {$A$};
 \draw node at (8,-1) {$B$};

    \end{tikzpicture}
    \end{figure}

In the figures we have, as in Case (1), indicated by the letters $A$ and $B$ how to stack the strips  together to form a puzzle. This gives uncountably many odd puzzles,  which are again parametrized by a transversal tree of order 3.

Finally, the twelve exceptional puzzles in case (3) can be characterized as being the unique odd puzzles which contain one and only one of the following twelve flat shapes:

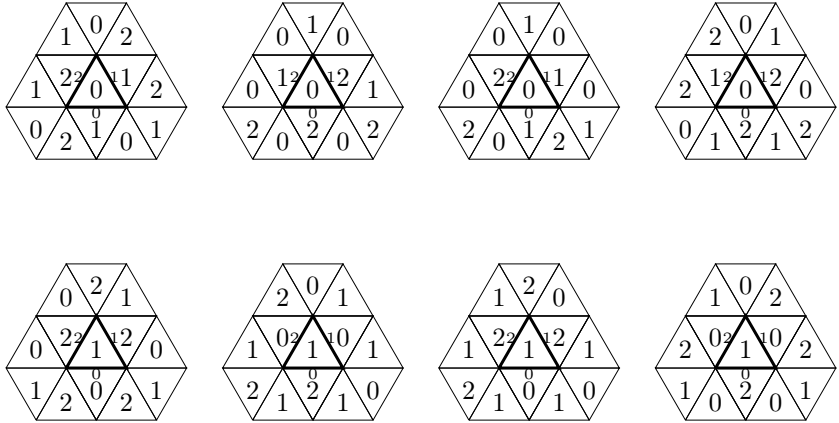
\begin{figure}[H]
 \begin{tikzpicture}[scale=.8]
 \upthick{0}{0}{}
 
     \draw node at (.5,-.1) {\tiny 0};
     \draw node at (.8,.5) {\tiny 1};
     \draw node at (0.2,.5) {\tiny 2};
     
\hexagon{1}{0}{2}{1}{0}{1}{0}{1}
\hexagon{.5}{1}{2}{0}{1}{2}{}{}
\hexagon{0}{0}{}{}{1}{0}{2}{}
    \end{tikzpicture}\ \ \ \ 
     \begin{tikzpicture}[scale=.8]
 \upthick{0}{0}{}
 
    \draw node at (.5,-.1) {\tiny 0};
     \draw node at (.8,.5) {\tiny 1};
     \draw node at (0.2,.5) {\tiny 2};
      
\hexagon{1}{0}{1}{2}{0}{2}{0}{2}
\hexagon{.5}{1}{0}{1}{0}{1}{}{}
\hexagon{0}{0}{}{}{0}{2}{0}{}
    \end{tikzpicture}\ \ \ \ 
 \begin{tikzpicture}[scale=.8]
 \upthick{0}{0}{}
 
     \draw node at (.5,-.1) {\tiny 0};
     \draw node at (.8,.5) {\tiny 1};
     \draw node at (0.2,.5) {\tiny 2};
     
\hexagon{1}{0}{0}{1}{0}{1}{2}{1}
\hexagon{.5}{1}{0}{1}{0}{2}{}{}
\hexagon{0}{0}{}{}{0}{2}{0}{}
    \end{tikzpicture}\ \ \ \ 
     \begin{tikzpicture}[scale=.8]
 \upthick{0}{0}{}
 
    \draw node at (.5,-.1) {\tiny 0};
     \draw node at (.8,.5) {\tiny 1};
     \draw node at (0.2,.5) {\tiny 2};
      
\hexagon{1}{0}{0}{2}{0}{2}{1}{2}
\hexagon{.5}{1}{1}{0}{2}{1}{}{}
\hexagon{0}{0}{}{}{2}{0}{1}{}
    \end{tikzpicture}
     \end{figure}

\begin{figure}[H]
 \begin{tikzpicture}[scale=.8]
 \upthick{0}{0}{}
 
     \draw node at (.5,-.1) {\tiny 0};
     \draw node at (.8,.5) {\tiny 1};
     \draw node at (0.2,.5) {\tiny 2};
     
\hexagon{1}{0}{0}{2}{1}{0}{2}{1}
\hexagon{.5}{1}{1}{2}{0}{2}{}{}
\hexagon{0}{0}{}{}{0}{1}{2}{}
    \end{tikzpicture}\ \ \ \ 
     \begin{tikzpicture}[scale=.8]
 \upthick{0}{0}{}
 
    \draw node at (.5,-.1) {\tiny 0};
     \draw node at (.8,.5) {\tiny 1};
     \draw node at (0.2,.5) {\tiny 2};
      
\hexagon{1}{0}{1}{0}{1}{2}{1}{0}
\hexagon{.5}{1}{1}{0}{2}{0}{}{}
\hexagon{0}{0}{}{}{1}{2}{1}{}
    \end{tikzpicture}\ \ \ \ 
     \begin{tikzpicture}[scale=.8]
 \upthick{0}{0}{}
 
     \draw node at (.5,-.1) {\tiny 0};
     \draw node at (.8,.5) {\tiny 1};
     \draw node at (0.2,.5) {\tiny 2};
     
\hexagon{1}{0}{1}{2}{1}{0}{1}{0}
\hexagon{.5}{1}{0}{2}{1}{2}{}{}
\hexagon{0}{0}{}{}{1}{2}{1}{}
    \end{tikzpicture}\ \ \ \ 
     \begin{tikzpicture}[scale=.8]
 \upthick{0}{0}{}
 
    \draw node at (.5,-.1) {\tiny 0};
     \draw node at (.8,.5) {\tiny 1};
     \draw node at (0.2,.5) {\tiny 2};
      
\hexagon{1}{0}{2}{0}{1}{2}{0}{1}
\hexagon{.5}{1}{2}{0}{1}{0}{}{}
\hexagon{0}{0}{}{}{2}{1}{0}{}
    \end{tikzpicture}

     \end{figure}

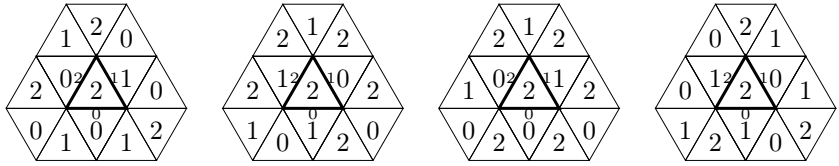
\begin{figure}[H]
 \begin{tikzpicture}[scale=.8]
 \upthick{0}{0}{}
 
     \draw node at (.5,-.1) {\tiny 0};
     \draw node at (.8,.5) {\tiny 1};
     \draw node at (0.2,.5) {\tiny 2};
     
\hexagon{1}{0}{0}{1}{2}{0}{1}{2}
\hexagon{.5}{1}{0}{2}{1}{0}{}{}
\hexagon{0}{0}{}{}{2}{0}{1}{}
    \end{tikzpicture}\ \ \ \  
     \begin{tikzpicture}[scale=.8]
 \upthick{0}{0}{}
 
    \draw node at (.5,-.1) {\tiny 0};
     \draw node at (.8,.5) {\tiny 1};
     \draw node at (0.2,.5) {\tiny 2};
      
\hexagon{1}{0}{2}{0}{2}{1}{2}{0}
\hexagon{.5}{1}{2}{1}{2}{1}{}{}
\hexagon{0}{0}{}{}{2}{1}{0}{}
    \end{tikzpicture}\ \ \ \ 
     \begin{tikzpicture}[scale=.8]
 \upthick{0}{0}{}
 
     \draw node at (.5,-.1) {\tiny 0};
     \draw node at (.8,.5) {\tiny 1};
     \draw node at (0.2,.5) {\tiny 2};
     
\hexagon{1}{0}{2}{1}{2}{0}{2}{0}
\hexagon{.5}{1}{2}{1}{2}{0}{}{}
\hexagon{0}{0}{}{}{1}{0}{2}{}
    \end{tikzpicture}\ \ \ \ 
     \begin{tikzpicture}[scale=.8]
 \upthick{0}{0}{}
 
    \draw node at (.5,-.1) {\tiny 0};
     \draw node at (.8,.5) {\tiny 1};
     \draw node at (0.2,.5) {\tiny 2};
      
\hexagon{1}{0}{1}{0}{2}{1}{0}{2}
\hexagon{.5}{1}{1}{2}{0}{1}{}{}
\hexagon{0}{0}{}{}{0}{1}{2}{}
    \end{tikzpicture}
\caption{The twelve special odd puzzles} 
     \end{figure}

It is not difficult to verify that each one of these twelves flats shapes belong to a unique odd ring puzzle, and that the twelve resulting puzzles are pairwise non-isomorphic (where by isomorphism we mean a label preserving simplicial isomorphism).

\section{The direct approach}\label{S - Meso}

We shall begin with a description of the direct, and most natural approach to the odd ring puzzle problem. Unfortunately, it is inconclusive, and we shall explain the reason why in the present section.

As mentioned in the introduction, it is always possible to start with the following triangle having a zero marking in its interior: 

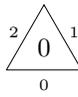
\begin{figure}[H]
 \begin{tikzpicture}
 \up{0}{0}{0}
     \draw node at (.5,-.2) {\tiny 0};
     \draw node at (.9,.5) {\tiny 1};
     \draw node at (0.1,.5) {\tiny 2};
    \end{tikzpicture}
     \end{figure}

\noindent 

We have indicated the initial edge marking from which the labels on every edge of the tessellation are determined (see Lemma \ref{L - edge labelling}). The objective is  to ``build around this triangle'' and attempt to classify all odd puzzles.  This is a legitimate approach and it will classify all puzzles if it succeeds (since it is clear that the above triangle belongs to any odd puzzle). 

We first remark that, by definition of the ring set (vertex marking), each vertex in the tessellation can only apply three of the nine rings given by $\Theta$. The present section explains how that is due to the so-called ``mesoscopic rank property''. 

Building around the above triangle, we observe that every side can be extended in precisely two ways. Therefore there are in total eight possible initial configurations for an odd ring puzzle. The configurations are: 

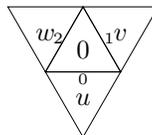
\begin{figure}[H]
\begin{tikzpicture}
\figureuvw{$u$}{$v$}{$w$}
    \draw node at (.5,-.1) {\tiny 0};
     \draw node at (.85,.45) {\tiny 1};
     \draw node at (0.17,.45) {\tiny 2};
 \end{tikzpicture}
\caption{Eight configurations around the initial triangle (here $u,v,w\in \{1,2\}$)}
\end{figure}

The three subwords on the bourdary are: $^0v^10^0u^1$, $^1w^20^1v^2$, and $^2u^00^2w^0$, where we have indicated the edge labelings as exponents. We shall make three remarks:

(a) The first word $^0v^10^0u^1$ necessarily appears in a ring belonging to $\Theta_0$. In fact, it appears twice in each ring. For the first ring (on the left), the two cases are $u=1,v=2$ and $u=2, v=1$, for the second ring, both cases occurs with $u=v=1$, and for the last ring both cases occurs with $u=v=2$.

(b) The second word necessarily appears in ring  belonging to $\Theta_1$. It appears twice in the first ring ($w=2$, $v=1$ or $w=1$, $v=2$), once in the second ring ($v=w=2$), and three times in the last ring ($w=1, v=2$ or $w=2, v=1$   or $w=v=1$).

(c) The third word necessarily appears in ring  belonging to $\Theta_2$. It appears twice in the first ring ($u=1$, $w=2$ or $u=2$, $w=1$), three times in the second ring ($u=2, w=1$, or $u=1, w=2$, or $u=w=2$), and once in the last ring ($u=w=1$).

Let us do the case  $u=v=w=1$, for an example. This case is initially simpler, since the triangle
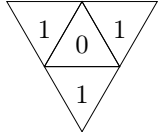
\begin{figure}[H]
\begin{tikzpicture}
\figureuvw{1}{1}{1}
\end{tikzpicture}
\end{figure}

\noindent can only be extended in two ways (as opposed to eight, in general), as follows: 

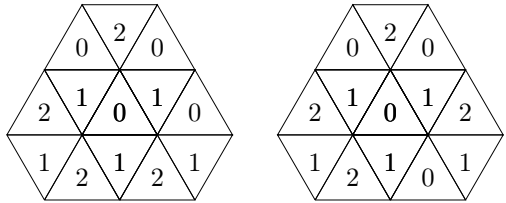
\begin{figure}[H]

\begin{tikzpicture}
\hexagon{0}{0}{0}{1}{2}{1}{2}{1}
\hexagon{1}{0}{0}{1}{0}{1}{2}{1}
\hexagon{.5}{1}{0}{2}{0}{1}{0}{1}

\end{tikzpicture}\ \ \ \ \
\begin{tikzpicture}
\hexagon{0}{0}{0}{1}{2}{1}{2}{1}
\hexagon{1}{0}{2}{1}{0}{1}{0}{1}
\hexagon{.5}{1}{0}{2}{0}{1}{0}{1}

\end{tikzpicture}

\end{figure} 

Consider for instance the left hand case. Here there are four possible extensions, because the extension is unique at the top right of the figure. In fact, the extensions are, precisely:

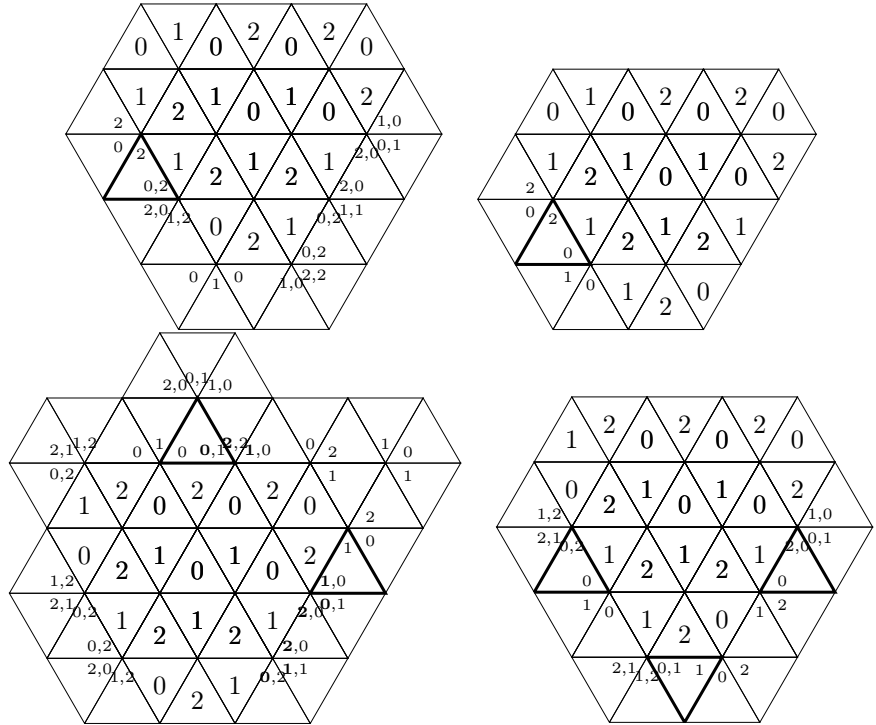
\begin{figure}[H]

\begin{tikzpicture}
\hexagon{0}{0}{0}{1}{2}{1}{2}{1}
\hexagon{1}{0}{0}{1}{0}{1}{2}{1}
\hexagon{.5}{1}{0}{2}{0}{1}{0}{1}

\hexagon{1.5}{1}{0}{2}{0}{1}{0}{2}
\hexagon{-.5}{1}{0}{1}{0}{1}{2}{1}
\hexagon{.5}{-1}{2}{1}{2}{0}{2}{1}

\hexagonsmall{-1}{0}{}{}{2}{0}{2}{}
\hexagonsmall{-.5}{-1}{}{}{0,2}{2,0}{1,2}{}

\hexagonsmall{2}{0}{1,0}{}{}{}{2,0}{0,1}
\hexagonsmall{1.5}{-1}{2,0}{}{}{}{0,2}{1,1}

\hexagonsmall{0}{-2}{}{}{}{0}{1}{0}
\hexagonsmall{1}{-2}{0,2}{}{}{}{1,0}{2,2}

\upthick{-1.5}{-1}

\end{tikzpicture}
\ \ \ 
\begin{tikzpicture}
\hexagon{0}{0}{0}{1}{2}{1}{2}{1}
\hexagon{1}{0}{0}{1}{0}{1}{2}{1}
\hexagon{.5}{1}{0}{2}{0}{1}{0}{1}

\hexagon{1.5}{1}{0}{2}{0}{1}{0}{2}
\hexagon{-.5}{1}{0}{1}{0}{1}{2}{1}
\hexagon{.5}{-1}{2}{1}{2}{1}{2}{0}

\hexagonsmall{-1}{0}{}{}{2}{0}{2}{}
\hexagonsmall{-.5}{-1}{}{}{0}{1}{0}{}



\upthick{-1.5}{-1}

\end{tikzpicture}

\begin{tikzpicture}
\hexagon{0}{0}{0}{1}{2}{1}{2}{1}
\hexagon{1}{0}{0}{1}{0}{1}{2}{1}
\hexagon{.5}{1}{0}{2}{0}{1}{0}{1}

\hexagon{1.5}{1}{0}{2}{0}{1}{0}{2}
\hexagon{-.5}{1}{0}{2}{1}{0}{2}{1}
\hexagon{.5}{-1}{2}{1}{2}{0}{2}{1}

\hexagonsmall{-1}{0}{}{}{1,2}{2,1}{0,2}{}
\hexagonsmall{-.5}{-1}{}{}{0,2}{2,0}{1,2}{}

\hexagonsmall{2}{0}{\textbf{1},0}{}{}{}{\textbf{2},0}{\textbf{0},1}
\hexagonsmall{1.5}{-1}{\textbf{2},0}{}{}{}{\textbf{0},2}{\textbf{1},1}


\hexagonsmall{0}{2}{0}{1}{0}{}{}{}
\hexagonsmall{1}{2}{\textbf{1},0}{\textbf{2},2}{\textbf{0},1}{}{}{}
\upthick{0}{2}


\hexagonsmall{-1}{2}{}{1,2}{2,1}{0,2}{}{}

\hexagonsmall{.5}{3}{1,0}{0,1}{2,0}{}{}{}


\hexagonsmall{2}{2}{2}{0}{}{}{}{1}
\hexagonsmall{2.5}{1}{2}{}{}{}{1}{0}

\upthick{2}{0}

\hexagonsmall{3}{2}{0}{1}{}{}{}{1}

\end{tikzpicture}
\ \ \ 
\begin{tikzpicture}
\hexagon{0}{0}{0}{1}{2}{1}{2}{1}
\hexagon{1}{0}{0}{1}{0}{1}{2}{1}
\hexagon{.5}{1}{0}{2}{0}{1}{0}{1}

\hexagon{1.5}{1}{0}{2}{0}{1}{0}{2}
\hexagon{-.5}{1}{0}{2}{1}{0}{2}{1}
\hexagon{.5}{-1}{2}{1}{2}{1}{2}{0}

\hexagonsmall{-1}{0}{}{}{1,2}{2,1}{0,2}{}
\hexagonsmall{-.5}{-1}{}{}{0}{1}{0}{}

\hexagonsmall{2}{0}{1,0}{}{}{}{2,0}{0,1}
\hexagonsmall{1.5}{-1}{0}{}{}{}{1}{2}

\hexagonsmall{0}{-2}{}{}{}{2,1}{1,2}{0,1}
\hexagonsmall{1}{-2}{}{}{}{1}{0}{2}

\upthick{1.5}{-1}

\upthick{-1.5}{-1}

\downthick{.5}{-3}

\end{tikzpicture}
\caption{Four possible extensions}\label{Fig - 111 4 possible ext}
\end{figure}

In figure \ref{Fig - 111 4 possible ext}, we have extended all drawings analytically whenever possible. In some cases, there are two possible extensions,  which explains why some angles in the figures are endowed with a pair of labels. Extending around one angle of a triangle, or the other angle, may lead to markings that fit together in a unique way, in order to provide a partial extension into a strip of height 1. When this happens, we have shown in the figure which choice of label extends uniquely in bold font; the figure also contains triangles with bold sides, which indicate where the possible extensions fit together in at most one way.  The figures and computations shown above in fact only provide a contradiction in Fig.\ \ref{Fig - 111 4 possible ext}.2 (i.e., top right drawing), on the left hand side of this drawing, displayed as a bold triangle.   

It is apparent from even the initial combinatorial complexity that this approach can hardly be carried out to prove Theorem \ref{T - main theorem}. Conceptually, the reason why this fails is the fact that the unique odd Moebius--Kantor complex has  the exponential mesoscopic rank property (this was established in \cite{rd} in the even case, but a similar argument holds in the odd case). For a Moebius--Kantor complex, the exponential mesoscopic rank property means that the number of Euclidean disks of a given radius, with fixed center, which embed in the given complex but fail to embed into a flat plane of the complex, grows exponentially fast in the radius. These disks  provide partial puzzles for the odd puzzle problem (as in Fig.\ \ref{Fig - 111 4 possible ext}) which are not relevant for our classification, since they are ultimately not embeddable in a full puzzle (i.e., tessellation of the whole Euclidean plane).  An exhaustive classification (not required to establish Th.\ \ref{T - main theorem}) would involve a description of all mesoscopic puzzles---by which we mean all partial puzzles (or at least the maximal ones) which can be formed from the initial set of shapes and rings of the puzzle problem, but do not embed as a subset of a full odd puzzle. Fig.\ \ref{Fig - 111 4 possible ext}.2 is an example of such a puzzle (not a maximal one). It corresponds to a mesoscopic flat disk in the odd Moebius--Kantor complex.

\section{Roots in CAT(0) complexes and ring puzzles}\label{S - Roots}

We shall propose an alternative approach which solves the mesoscopic rank issues raised in the previous section. Our goal is to establish a full classification of all odd puzzles (which are tessellations of the full Euclidean plane) without having to classify the partial (mesoscopic) puzzles.    This relies on concepts in intermediate rank geometry, and we have to recall some definitions first. We also define precisely what Moebius--Kantor complexes are, and the notions of parity, roots, and root distributions.
 
Let $\Delta$ be a nonpositively curved 2-complex (a general reference on nonpositive curvature is \cite{BH}).  If $x\in \Delta$ is a vertex, we write $L_x$ for the link at $x$, endowed with the angular metric. Thus, $L_x$ is a metric graph. We always assume that $\Delta$ is a locally finite complex (in fact, $\Delta$ will only denote the unique odd Moebius--Kantor complex in this paper). 

\begin{definition} We call a \emph{root} at $x$ an isometric embedding $\alpha\colon [0,\pi]\inj L_x$, such that $\alpha(0)$ is a link vertex of degree greater than $2$.
\end{definition}

Every root has a \emph{rank} denoted $\rk(\alpha)$. It is a rational number defined by
\[
\rk(\alpha):=1+{N(\alpha)\over q_\alpha}\in [1,2]
\] 
where 
\[
N(\alpha):=|\{\beta\in \Phi_x\mid \alpha\neq \beta, \alpha(0)=\beta(0), \alpha(\pi)=\beta(\pi)\}|,
\]
$\Phi_x$ denotes the set of roots at $x$ and, for a root $\alpha$,  $q_\alpha$ denotes the degree of $\alpha(0)$ minus 1:
\[
q_\alpha:=\val(\alpha(0))-1.
\]  
We refer to \cite[\ts 4]{chambers} for more details on these definitions.

\begin{definition}
A Moebius--Kantor complex is a CAT(0) 2-complex with equilateral triangle faces in which every link is isomorphic to the Moebius--Kantor graph.
\begin{figure}[H]
\begin{tikzpicture}[scale=.5,line join=bevel,z=-5.5]
\tikzstyle{every node}=[font=\tiny]

\coordinate (H0) at (3,0);
\coordinate (H1) at (2.771,1.148);
\coordinate (H2) at (2.121,2.121);
\coordinate (H3) at (1.148,2.771);
\coordinate (H4) at (0,3);
\coordinate (H5) at (-1.149,2.771);
\coordinate (H6) at (-2.122,2.121);
\coordinate (H7) at (-2.772,1.148);
\coordinate (H8) at (-3,0);
\coordinate (H9) at (-2.772,-1.149);
\coordinate (H10) at (-2.122,-2.122);
\coordinate (H11) at (-1.149,-2.772);
\coordinate (H12) at (-0.001,-3);
\coordinate (H13) at (1.148,-2.772);
\coordinate (H14) at (2.121,-2.122);
\coordinate (H15) at (2.771,-1.149);
\draw[solid,line width=0.1mm,color=black,-] (2.771,1.148) -- (3,0);
\draw[solid,line width=0.1mm,color=black,-] (2.771,1.148) -- (3,0);
\draw[solid,line width=0.1mm,color=black,-] (2.121,2.121) -- (2.771,1.148);
\draw[solid,line width=0.1mm,color=black,-] (2.121,2.121) -- (2.771,1.148);
\draw[solid,line width=0.1mm,color=black,-] (1.148,2.771) -- (2.121,2.121);
\draw[solid,line width=0.1mm,color=black,-] (1.148,2.771) -- (2.121,2.121);
\draw[solid,line width=0.1mm,color=black,-] (0,3) -- (1.148,2.771);
\draw[solid,line width=0.1mm,color=black,-] (0,3) -- (1.148,2.771);
\draw[solid,line width=0.1mm,color=black,-] (-1.149,2.771) -- (0,3);
\draw[solid,line width=0.1mm,color=black,-] (-1.149,2.771) -- (0,3);
\draw[solid,line width=0.1mm,color=black,-] (-1.149,2.771) -- (3,0);
\draw[solid,line width=0.1mm,color=black,-] (-1.149,2.771) -- (3,0);
\draw[solid,line width=0.1mm,color=black,-] (-2.122,2.121) -- (-1.149,2.771);
\draw[solid,line width=0.1mm,color=black,-] (-2.122,2.121) -- (-1.149,2.771);
\draw[solid,line width=0.1mm,color=black,-] (-2.772,1.148) -- (-2.122,2.121);
\draw[solid,line width=0.1mm,color=black,-] (-2.772,1.148) -- (-2.122,2.121);
\draw[solid,line width=0.1mm,color=black,-] (-2.772,1.148) -- (2.121,2.121);
\draw[solid,line width=0.1mm,color=black,-] (-2.772,1.148) -- (2.121,2.121);
\draw[solid,line width=0.1mm,color=black,-] (-3,0) -- (-2.772,1.148);
\draw[solid,line width=0.1mm,color=black,-] (-3,0) -- (-2.772,1.148);
\draw[solid,line width=0.1mm,color=black,-] (-2.772,-1.149) -- (-3,0);
\draw[solid,line width=0.1mm,color=black,-] (-2.772,-1.149) -- (-3,0);
\draw[solid,line width=0.1mm,color=black,-] (-2.772,-1.149) -- (0,3);
\draw[solid,line width=0.1mm,color=black,-] (-2.772,-1.149) -- (0,3);
\draw[solid,line width=0.1mm,color=black,-] (-2.122,-2.122) -- (-2.772,-1.149);
\draw[solid,line width=0.1mm,color=black,-] (-2.122,-2.122) -- (-2.772,-1.149);
\draw[solid,line width=0.1mm,color=black,-] (-1.149,-2.772) -- (-2.122,-2.122);
\draw[solid,line width=0.1mm,color=black,-] (-1.149,-2.772) -- (-2.122,-2.122);
\draw[solid,line width=0.1mm,color=black,-] (-1.149,-2.772) -- (-2.122,2.121);
\draw[solid,line width=0.1mm,color=black,-] (-1.149,-2.772) -- (-2.122,2.121);
\draw[solid,line width=0.1mm,color=black,-] (-0.001,-3) -- (-1.149,-2.772);
\draw[solid,line width=0.1mm,color=black,-] (-0.001,-3) -- (-1.149,-2.772);
\draw[solid,line width=0.1mm,color=black,-] (-0.001,-3) -- (2.771,1.148);
\draw[solid,line width=0.1mm,color=black,-] (-0.001,-3) -- (2.771,1.148);
\draw[solid,line width=0.1mm,color=black,-] (1.148,-2.772) -- (-0.001,-3);
\draw[solid,line width=0.1mm,color=black,-] (1.148,-2.772) -- (-0.001,-3);
\draw[solid,line width=0.1mm,color=black,-] (1.148,-2.772) -- (-3,0);
\draw[solid,line width=0.1mm,color=black,-] (1.148,-2.772) -- (-3,0);
\draw[solid,line width=0.1mm,color=black,-] (2.121,-2.122) -- (1.148,-2.772);
\draw[solid,line width=0.1mm,color=black,-] (2.121,-2.122) -- (1.148,-2.772);
\draw[solid,line width=0.1mm,color=black,-] (2.121,-2.122) -- (1.148,2.771);
\draw[solid,line width=0.1mm,color=black,-] (2.121,-2.122) -- (1.148,2.771);
\draw[solid,line width=0.1mm,color=black,-] (2.771,-1.149) -- (2.121,-2.122);
\draw[solid,line width=0.1mm,color=black,-] (2.771,-1.149) -- (2.121,-2.122);
\draw[solid,line width=0.1mm,color=black,-] (2.771,-1.149) -- (3,0);
\draw[solid,line width=0.1mm,color=black,-] (2.771,-1.149) -- (3,0);
\draw[solid,line width=0.1mm,color=black,-] (2.771,-1.149) -- (-2.122,-2.122);
\draw[solid,line width=0.1mm,color=black,-] (2.771,-1.149) -- (-2.122,-2.122);
\end{tikzpicture}
\caption{The Moebius--Kantor graph}\label{Fig MK Link}
 \end{figure}
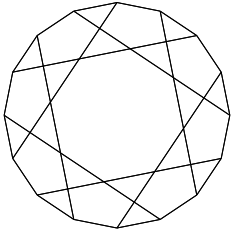
\end{definition}

In a Moebius--Kantor complex, the rank of a root can be either $\frac 3 2$ or $2$. Both types of roots are found in the obvious Hamiltonian cycle of length 16 shown in Fig.\ \ref{Fig MK Link}. Since this graph is 2-transitive, it is obvious that there can only be two distinct orbits of roots, and their rank is straightforward to compute. 

The parity of a face  in a Moebius--Kantor complex $\Delta$ is defined as follows (cf.\  \cite[\ts 2]{parity}). Let $f$ be a face in $\Delta$.  Let $\tilde f$ denote the equilateral triangle formed by the union of $f$ and three faces in $\Delta$ corresponding to a choice of a face  not equal to $f$ adjacent to every side of $f$. We call $\tilde f$ a \emph{large triangle} containing $f$.
 For every vertex $x$ of $f$, the triangle $\tilde f$ determines a root $\alpha_x$ at $x$ in $\Delta$. 

\begin{definition} We call the \emph{parity} of $f$ the parity of the number of roots $\alpha_x$ which are of rank $2$, when $x$ runs over the three vertices of the face $f$.
\end{definition}

This is a well-defined invariant, i.e., the parity of $f$ does not depend on the choice of the large triangle $\tilde f$ containing $f$, by \cite[Lemma 2.1]{parity}. 

\begin{definition}
We say that $\Delta$ is odd (resp., even) if every face in $\Delta$ is odd (resp., even).
\end{definition}

We have shown recently in \cite{Sidon} that there exists a unique odd Moebius--Kantor complex.

It is not difficult to extend these notions to ring puzzles. Namely, let $P=(\Sigma,\Theta)$ be a ring puzzle problem (defined by two sets $\Sigma$ of shapes and $\Theta$ of rings). We shall asssume here to simplify  that $P$ is a marked triangle puzzle problem. Thus, each element in $\Theta$ is a marked circle of length $2\pi$ subdivided equally into six arcs. Every vertex (resp.\ edge) and every edge (resp.\ interior) of every ring (resp.\ shape) is marked such that that adjacent ring edges have distinct markings.

We call root of $P$ a marked simplicial isometric embedding $\alpha\colon D\inj R$, where $D$ is a marked segment of length $\pi$ subdivided into  three arcs, called the domain of $\alpha$, and $R\in \Theta$ is a ring. For instance, the unique marked embedding of the segment $D$ defined by
\begin{figure}[H]
\begin{tikzpicture}[scale=1.4]
\tikzstyle{every node}=[font=\small]

\draw[solid,line width=0.1mm,color=black,-] (0,0) -- (3,0);
\draw[solid,line width=0.1mm,color=black,-] (0,.1) -- (0,-.1);
\draw[solid,line width=0.1mm,color=black,-] (1,.1) -- (1,-.1);
\draw[solid,line width=0.1mm,color=black,-] (2,.1) -- (2,-.1);
\draw[solid,line width=0.1mm,color=black,-] (3,.1) -- (3,-.1);

     \draw node at (0,-.2) {$s$};
     \draw node at (0,.2) {$0$};
     \draw node at (.5,.2) {$s+1$};
     \draw node at (1,-.2) {$s+1$};
     \draw node at (1.5,.2) {$s$};
     \draw node at (2,-.2) {$s$};
     \draw node at (2.5,.2) {$s+2$};
     \draw node at (3,-.2) {$s+1$};
     \draw node at (3,.2) {$\pi$};

\end{tikzpicture}
\end{figure}
\noindent in the first ring of $\Theta_s$ is a root of the odd ring puzzle problem. Let $\Phi_P$ be the set of roots of $P$.

We write $\sim$ for the equivalence relation on $\Phi_P$ generated by $\alpha\sim \beta$ if and only if $\dom \alpha=\dom \beta$ (both roots have an identical domain) or $\mathop{\mathrm{Im}} \alpha=\mathop{\mathrm{Im}}\beta$ (both roots sit in a common ring).  If $\alpha\in \Phi_P$, we call multiplicity of $\alpha$ the number $N_P(\alpha)$ of roots $\beta\sim \alpha$ such that  $\dom \alpha=\dom \beta$. 

We say that a root $\beta$ is  \emph{complement} to a root $\alpha$ if $\alpha(0)=\beta(0)$ and there exists a ring $R\in \Theta$ such that $\mathop{\mathrm{Im}} \alpha\cup \mathop{\mathrm{Im}} \beta=R$. We say that $P$ has the \emph{$\theta_0$-extension property}  if whenever a marked segment $[0,\theta]$ of length $\theta\in [0,2\pi]$ can be embedded in distinct rings  (isometrically and respecting the marking), then $\theta\leq \theta_0$. We say that $P$ is nonpositively curved if it does not have the $(\pi+\e)$-extension for some $\e>0$. The odd ring puzzle problem is nonpositively curved in this sense.   
If $\alpha$ is a root in a nonpositively curved ring puzzle problem, we let $q_\alpha$ denote the order of $\alpha(0)$ minus 1, where the order of $\alpha(0)$ is number of edge markings which are adjacent, in some ring, to the marking of $\alpha(0)$.  If furthermore $\alpha$ has multiplicity $N$, then there are precisely $N$ roots $\beta$ which are complement to a root with domain $\dom(\alpha)$, and therefore $N\leq q_\alpha$.  In this case we define the rank of $\alpha$ to be:
\[
\rk(\alpha):=1+{N_P(\alpha)\over q_\alpha}\in (1,2]
\] 
where: $N_P(\alpha):=|\{\beta\in \Phi_P\mid \dom \alpha=\dom \beta\}|$. With these definitions, in the odd ring puzzle problem, every root has rank 2, except for the roots $\alpha$ in which, in Fig.\ \ref{Fig 2}, $\alpha(0)$ and $\alpha(\pi)$ are positioned horizontally. In this case the rank is $\frac 3 2$, and this corresponds to what happens in the unique odd Moebius--Kantor complex.

\section{Classifying the odd root distributions}

    We shall classify the odd root distributions first. By definition, a \emph{root distribution} on the equilateral triangle lattice is a function which assigns to every vertex in the tessellation one of the three simplicial directions. General root distributions in this sense are studied in \cite{parity} and are related to Moebius--Kantor geometry. By definition, the root distribution associated with an odd ring puzzle takes a vertex to the unique direction determined by the root of rank $\frac 3 2$ in the ring positioned at this vertex. It follows, by the definition of odd ring puzzles (rings in $\Theta_s$), that the other two directions at this vertex correspond to roots of rank 2.  If $t$ is a triangle in the lattice, we call as before large triangle associated with $t$ the triangle with sides of length 2 which contains $t$ in its center as well as the three triangles of the lattice which are adjacent to $t$. The three sides of this large triangle determine three simplicial directions, one at each vertex in $t$. The triangle $t$ is then called odd (given a root distribution) if and only if the number of such directions which are of rank 2 is odd. A root distribution is odd if every lattice triangle is odd.

Let $D$ be an odd root distribution. In what follows, we call geodesic (resp. half-geodesic) a bi-infinite (resp.\ semi-infinite) simplicial segment in the equilateral triangle lattice. Given any (possibly infinite) simplicial segment $s$ in the lattice and a vertex $v$ on $s$, we say that a root of rank $\frac{3}{2}$ at $v$ (i.e., the simplicial direction selected by $D$) is horizontal if it is parallel to $s$. If all roots of rank $\frac{3}{2}$ on a $s$ are horizontal, we say that $s$ itself is of rank $\frac{3}{2}$; if none of them are, then we say that $s$ is of rank 2. 

We call a (possibly infinite) simplicial segment  \emph{1-periodic} (with respect to $D$) if its vertex roots are parallel; if the roots are not parallel, but the roots at distance $2$ are, then we call the segment \emph{2-periodic}. 
 Finally, we call $1$-periodic (resp.\ $2$-periodic) rank $2$  strip any simplicial flat strip in which both boundaries are  1-periodic (resp.\ $2$-periodic) and have rank 2.
 
 We now proceed with the classification.

\begin{lemma}[1-periodic geodesics]\label{L - 1 periodic geodesic}
Every odd root distribution containing a 1-periodic geodesic is a union of 1-periodic rank 2 strips of height 2. Furthermore, a 1-periodic geodesic in an odd root distribution is of rank $\frac{3}{2}$ if and only if geodesics at distance 1 from it are 1-periodic of rank 2. 
\end{lemma}

\begin{proof}
Let $\gamma$ be a (horizontal) 1-periodic geodesic. Then, $\gamma$ is either of rank $\frac{3}{2}$, or of rank 2 such that all of its vertices induce parallel roots of rank $\frac{3}{2}$. Let $\gamma'$ be a geodesic at distance 1 from $\gamma$. 

Suppose $\gamma$ is of rank $\frac{3}{2}$, and let $t$ be any triangle with two vertices $v'_1$, $v'_2$ on $\gamma'$ and a single vertex $v$ on $\gamma$. Since triangles adjacent to $t$ with two vertices on $\gamma$ are also odd, the roots of rank $\frac{3}{2}$ at $v'_i$ are non-horizontal. Now, suppose there are two vertices $x$, $y$ at distance 1 from each other on $\gamma'$ with non-parallel roots of rank $\frac{3}{2}$. Then, the triangle $t'$ with vertices $x$, $y$, and a single vertex $z$ on $\gamma$ turns out to be even, a contradiction. This shows that $\gamma'$ is 1-periodic of rank 2. 

Suppose now that $\gamma$ is 1-periodic of rank 2. Consider any triangle $t$ with two vertices on $\gamma$ and a single vertex $v'$ on $\gamma'$. Since $t$ is odd, the root of rank $\frac{3}{2}$ at $v'$ is horizontal. It follows that $\gamma'$ is of rank $\frac{3}{2}$.  

The above implies, by induction, that a root distribution containing a 1-periodic geodesic is a union of 1-periodic rank 2 strips of height 2. 
\end{proof}

\begin{lemma}[2-periodic geodesics]\label{L - 2 periodic geodesic}
Every odd root distribution containing a 2-periodic geodesic is a union of 2-periodic rank 2 strips of height 1. 
\end{lemma}

\begin{proof}
Let $\gamma$ be a (horizontal) 2-periodic geodesic. We have two configurations: either $\gamma$ is of rank 2, or it contains a root of rank $\frac{3}{2}$---and thus, infinitely many, by 2-periodicity. Let $\gamma'$ be a geodesic at distance $1$ from $\gamma$.

Suppose $\gamma$ is of rank 2. Consider a triangle $t$ with two vertices $v_1$, $v_2$ on $\gamma$ and a single vertex $v'$ on $\gamma'$. Since $\gamma$ is 2-periodic of rank 2, the roots of rank $\frac{3}{2}$ at $v_1$ and $v_2$ are non-parallel and non-horizontal. Now, since $t$ is odd, this implies that the root of rank $\frac{3}{2}$ at $v'$ is non-horizontal, and thus either forms an angle of $\frac{2\pi}{3}$ with the horizontal (counterclockwise) or an angle of $\frac{\pi}{3}$. Suppose the former case (the argument for the latter case is essentially the same). Since the distribution is assumed to be odd, the roots of rank $\frac{3}{2}$ at distance 1 and 2 from $v'$ on $\gamma'$ are analytically determined. In particular, the roots at distance 1 are parallel and are transverse to the root of rank $\frac{3}{2}$ at $v'$, and those at distance 2 are parallel to the root of rank $\frac{3}{2}$ at $v'$. It follows, by induction, that  $\gamma'$ is 2-periodic of rank 2, and so that the distribution is a union of 2-periodic strips of rank 2 and height 1. 

Suppose now that $\gamma$ contains a root of rank $\frac{3}{2}$. We may assume, by symmetry, that the non-horizontal roots on $\gamma$ form an angle of $\frac{2\pi}{3}$ with the horizontal. Consider a triangle with two vertices $v_1$, $v_2$ on $\gamma$, and a single vertex $v'$ on $\gamma'$. Since $\gamma$ 2-periodic and contains a root of rank $\frac{3}{2}$, $v_1$ and $v_2$ define either 1 or 2 roots of rank 2 in a large triangle in which $t$ embeds as a centerpiece. We let $t$ be a triangle of the former type. Since $t$ is odd, it follows that the root of rank $\frac{3}{2}$ at $v'$ is horizontal. Now since the distribution is assumed to be odd, the roots of rank $\frac{3}{2}$ at distance $1$ from $v'$ are analytically determined to be parallel to the non-horizontal roots on $\gamma$.  It follows, by induction, that  $\gamma'$ is 2-periodic containing a root of rank $\frac{3}{2}$ (it is in fact a copy of $\gamma$), and so that the distribution is a union of transverse 2-periodic strips of rank 2 and height 1.
\end{proof}

 We shall distinguish three cases  of odd root distributions:

\begin{enumerate}
\item there exists a rank $\frac{3}{2}$ geodesic;
\item there exists a maximal rank $\frac{3}{2}$ half-geodesic;
\item every rank $\frac{3}{2}$ segment is bounded.
\end{enumerate}
The distributions in Case (1) are precisely described by Lemma \ref{L - 1 periodic geodesic}.  The distributions in  Case (2) contain distributions from Lemma \ref{L - 2 periodic geodesic} and a unique additional root distribution (see Lemma \ref{L - unique odd puzzle}). Finally, we shall conclude that Case (3) only contains distributions from Lemma \ref{L - 2 periodic geodesic}.

\begin{lemma}\label{L3}
Every odd root distribution contains a rank $\frac{3}{2}$ segment of length 3. 
\end{lemma}

In particular, a root distribution contains either $s_0$ (defined in Lemma \ref{L4}), or a rank $\frac{3}{2}$ geodesic (in which case Lemma \ref{L - 1 periodic geodesic} applies).

\begin{proof}
Since every triangle is odd, the number of roots of rank 2 in a large triangle is either 1 or 3. If there exists a triangle with a single root of rank 2, the side of one of the two triangles of base the segment associated with the roots of rank $\frac{3}{2}$ contains two aligned roots of rank $\frac{3}{2}$, and threfore a rank $\frac{3}{2}$ segment of length 3. 

It is clear that not every odd triangle contains exactly 3 roots of rank 2 such that no two of the roots of rank $\frac{3}{2}$ are not aligned. 
\end{proof}

\begin{lemma}\label{L4} Every (horizontal) segment $s_0=[a_0,b_0,c_0,d_0,e_0]$ of length 4 with root configuration $(\frac{3}{2},\frac{3}{2},2)$ in an odd root distribution admits a unique extension into a 2-periodic half-strip $S$ of height 4, which transversally alternates the root configurations $(\frac{3}{2},\frac{3}{2},2)$ and $(2,2,\frac{3}{2})$. Furthermore, the roots of rank $\frac{3}{2}$ in $S$ which are not horizontal are parallel.
\end{lemma}

\begin{proof}
We may assume, by symmetry, that the induced root of rank $\frac{3}{2}$ at $d_0$ forms an angle of $\frac{2\pi}{3}$ with the horizontal (counterclockwise). We show that $s_0$ extends downwards analytically to the left. For every $k\geq 1$ we let  $s_k=[a_k,b_k,c_k,d_k,e_k]$ be the horizontal segment below $s_{k-1}$ and $s_{-k}$ the segment opposite to it (we will adopt this notation for the remainder of the proof). Since the triangle $(c_0,d_0,d_1)$ is odd, the root of rank $\frac{3}{2}$ at $d_1$ is horizontal. Since $(c_0,c_1,d_1)$ and $(b_0,b_1,c_1)$ are odd, the roots of rank $\frac{3}{2}$ at $b_1$ and $c_1$ are parallel to the rank $\frac{3}{2}$ root at $d_0$. In particular, the root configuration of $s_1$ is $(2,2,\frac{3}{2})$.

Since the triangle $(b_1,c_1,c_2)$ is odd, the root at $c_2$ is horizontal. Since both triangles $(c_1,c_2,d_2)$ and $(c_1,d_1,d_2)$ are odd, the rank $\frac{3}{2}$ root at $d_3$ is parallel to the one at $d_0$.

Let us compute the position of the rank $\frac{3}{2}$ root at $b_2$. Since $(c_2,d_2,d_3)$ is odd, the rank $\frac{3}{2}$ root at $d_3$ is horizontal, which implies that the root at $c_3$ is parallel to the root at $d_0$. Now, since the triangles $(b_1,b_2,c_2)$ and $(b_2,c_2,c_3)$ are odd, the root of rank $\frac{3}{2}$  at $b_3$ is horizontal. Thus,  $s_2$ is a segment with root configuration $(\frac{3}{2},\frac{3}{2},2)$. The lemma follows by induction. 

The parallelogram of height 2 with boundaries $s_0$, $s_2$ and inner segment $s_1$ can be viewed as a ``glider'' which forces the construction of $S$ downward analytically.
\end{proof}

\bigskip

The proof of the lemma shows that if the root of rank $\frac{3}{2}$ at $d_0$ forms an angle of $\frac{2\pi}{3}$ with respect to the horizontal, the segment $s_0$ extends analytically in the downward direction. The extension is not analytic in the opposite direction. There are two cases: 
\begin{enumerate}
\item[(a)] $S$ can be extended into a 2-periodic bi-infinite strip of height 4, transversally alternating the root configurations $(\frac{3}{2},\frac{3}{2},2)$ and $(2,2,\frac{3}{2})$.
\item[(b)] $S$ admits a maximal extension into a periodic half-strip $S$ of height 4 with boundary $s_0$,  transversally alternating the root configurations $(\frac{3}{2},\frac{3}{2},2)$ and $(2,2,\frac{3}{2})$.
\end{enumerate}

Every root distribution in case (a) contains, by definition, a 2-periodic geodesic and therefore Lemma \ref{L - 2 periodic geodesic} applies. It remains to see what happens in case (b).

\bigskip

Let $t_0$ be the trapezoid of height $1$ and base $s_0$, in which the root of rank $\frac{3}{2}$ at $d_0$ forms an angle of $\frac{2\pi}{3}$ with respect to the horizontal (counterclockwise), and such that the root of rank $\frac{3}{2}$ at $c_{-1}$ forms an angle of $\frac{\pi}{3}$. The presence of $t_0$  breaks the 2-periodicity in the upward extension of $S$.  

\begin{lemma}
Every odd root distribution in case (b) contains $t_0$.
\end{lemma}

\begin{proof}
Let us denote again by $s_0=[a_0,b_0,c_0,d_0,e_0]$ the (horizontal) segment with root configuration $(\frac{3}{2},\frac{3}{2},2)$, and let $s_{-1}=[a_{-1},b_{-1},c_{-1},d_{-1},e_{-1}]$ be the segment above it in the direction opposite to $S$. By parity, there are two choices for the position of the root of rank $\frac{3}{2}$ at $c_{-1}$. One of the two choices gives $t_0$, the other gives a root forming an angle of $\frac{2\pi}{3}$  with respect to the horizontal. Since the triangles $(b_{-1},c_{-1},c_0)$ and $(b_0,c_0,b_{-1})$ are odd, the root of rank $\frac{3}{2}$ at $b_{-1}$ is parallel to the root at $d_0$. We let $s_{-2}=[a_{-2},b_{-2},c_{-2},d_{-2},e_{-2}]$ be the segment above $s_{-1}$ in the direction opposite to $S$. Since the triangle $(b_{-2},b_{-1},c_{-1})$ is odd, the root of rank $\frac{3}{2}$ at $b_{-2}$ is horizontal. Now, there are two choices for the position of the root of rank $\frac{3}{2}$ at $c_{-2}$. One of the two choices gives $t_0$, the other one is a horizontal root of rank $\frac{3}{2}$ which fails to happen by maximality of $S$. Indeed, the root at $d_{-1}$ would be horizontal, and therefore the root at $d_{-2}$  would  form an angle of $\frac{2\pi}{3}$  with respect to the horizontal. This would imply that the segments $s_{-1}$ and $s_{-2}$ have root configurations $(2,2,\frac{3}{2})$  and $(\frac{3}{2},\frac{3}{2},2)$ respectively, thus extending $S$. 
\end{proof}

\begin{lemma}\label{L - unique odd puzzle}
$t_0$ belongs to a unique odd root distribution $D_{0}$. 
\end{lemma}

\begin{proof}
Suppose we are given a copy of $t_0$ (which we assume is horizontal). We write again $s_0=[a_0,b_0,c_0,d_0,e_0]$ for the base of $t_0$ and $S$ for the strip in which $s_0$ extends downward analytically. For $k \geq 1$, we let as before $s_k=[a_k,b_k,c_k,d_k,e_k]$ be the segment  adjacent to $s_{k-1}$ included in $S$, and $s_{-k}$ be the adjacent segment opposite to $S$.  By assumption, the root at $c_{-1}$ forms an angle of $\frac{\pi}{3}$ with the horizontal (counterclockwise). Since the triangle $(c_{-1},d_0,d_{-1})$ is odd, the root of rank $\frac{3}{2}$ at $d_{-1}$ is parallel to the root of rank $\frac{3}{2}$ at $d_0$.  Similarly, since the triangle $(d_{-1},d_0,e_{0})$ is odd, the root of rank $\frac{3}{2}$ at $e_{0}$ is parallel to the root at $c_{-1}$.  
Furthermore, the root of rank $\frac{3}{2}$ at $d_1$ is horizontal by Lemma \ref{L4}. Since the triangles $(d_0,d_1,e_1)$ and $(d_0,e_0,e_1)$ are odd, the root of rank $\frac{3}{2}$ at  $e_1$ is parallel to the one at $d_0$. Thus, we get a copy of $s_0$ and so a transverse second copy $S'$ of the half-strip $S$ induced by the base of $t_0$. 

Since the triangle $(c_0,c_{-1},b_{-1})$ is odd, the root of rank $\frac{3}{2}$ at $b_{-1}$ is parallel to the one at $c_{-1}$. Similarly, since the triangle $(b_{-1},c_{-1},b_{-2})$ is odd, the root of rank $\frac{3}{2}$ at $b_{-2}$ is horizontal. Now, since the triangles $(c_{-1},b_{-2},c_{-2})$ and $(c_{-1},d_{-1},c_{-2})$ are odd, the root of rank $\frac{3}{2}$ at $c_{-2}$ is parallel to the one at $c_{-1}$. Therefore, we have another copy of $s_0$ and so a third copy $S''$ of $S$ transverse to $S$ and $S'$. 

By convexity, it follows that the trapezoid $t_0$ belongs to a unique root distribution, which we call $D_{0}$.
\end{proof}

Observe that $D_{0}$ contains a maximal rank $\frac{3}{2}$ half-geodesic in the upward extension of $S$, and so distributions in Case (3) all contain a 2-periodic geodesic. 

\bigskip

\begin{figure}[!h]
\begin{tikzpicture}[scale=0.75]
    \draw [line width=0.2mm] (180:0) -- +(240:3.5);
    \draw [line width=0.2mm] (180:1.5) -- +(240:3.5);
    \draw [line width=0.2mm] (300:0.75)+(180:0.375) -- +(0:3.5);
    \draw [line width=0.2mm] (120:0.75)+(180:0.375) -- +(0:3.5);
    \draw [line width=0.2mm] (180:1.125)+(300:0.375) -- +(120:3.5);
    \draw [line width=0.2mm] (90:0.9742785792)+(180:0.1875) -- +(120:3.5);
    \draw [dashed] (180:0.75) -- +(60:4.5);
\end{tikzpicture}
\caption{The three strips $S$, $S'$, and $S''$, along with the (dashed) maximal rank $\frac{3}{2}$ half-geodesic in the upward extension of $S$ (in the special distribution $D_{0}$)} 
\end{figure}
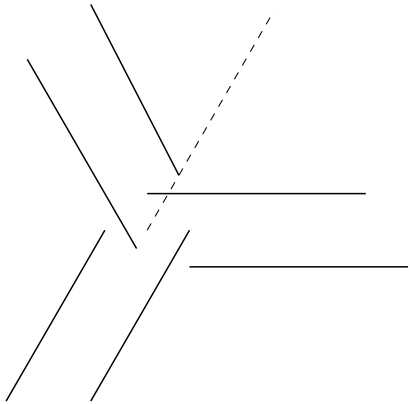

\begin{theorem}\label{t - odd root distributions}
An odd root distribution is either a union of 1-periodic rank 2 strips of height 2, or a union of 2-periodic rank 2 strips of height 1, or it is the special distribution  $D_{0}$. 
\end{theorem}

\begin{proof}
It follows by  Lemma \ref{L3} that an odd distribution $D$ contains a rank $\frac{3}{2}$ segment of length 3. In particular, it either contains a rank $\frac{3}{2}$ geodesic, or contains $s_0$ (defined in Lemma \ref{L4}). If $D$ contains a rank $\frac{3}{2}$ geodesic, then $D$ is a union of 1-periodic rank 2 strips of height 2 by Lemma \ref{L - 1 periodic geodesic}. If it does not, it contains $s_0$ and, by Lemma \ref{L4}, $D$ either contains a 2-periodic geodesic, in which case it is a union of 2-periodic rank 2 strips of height 1 by Lemma \ref{L - 2 periodic geodesic} (case (a)), or $D$ is the special distribution $D_{0}$ of Lemma \ref{L - unique odd puzzle} (case (b)). 
\end{proof}

\begin{remark}
(a) Given an odd distribution containing a 2-periodic rank 2 geodesic $\gamma$, it follows by the proof of Lemma \ref{L - 2 periodic geodesic} that any geodesic $\gamma'$ at distance $1$ from $\gamma$ is again 2-periodic of rank 2. In particular, its roots of rank $\frac{3}{2}$ are determined by the position of the root of rank $\frac{3}{2}$ at any given vertex on $\gamma'$ (it is, by parity, one of two possible positions). If $\eta$ is a geodesic transverse to $\gamma$, we can encode the position of the roots in a sequence $(x_n)_{n \in \mathbb{Z}}$ where $x_n$ indicates the chosen direction at vertex at distance $n\in \mathbb{Z}$ from the intersection point on $\gamma$. Aligning these roots suitably, we can in particular use 2-periodic rank 2 strips of height 1 to  construct an odd distribution with a full geodesic of rank $\frac{3}{2}$ as in Case (1),  one with a maximal rank $\frac{3}{2}$ half-geodesic as in Case (2), or one such that segments of rank $\frac{3}{2}$ have lengths uniformly bounded by a given constant $k >3$ as in Case (3) (this corresponds to letting $(x_n)_{n \in \mathbb{Z}}$ have no $k-2$ consecutive $0$s or $1$s).  Similarly, for a distribution obtained at a union of 1-periodic rank 2 strips of height 2, one may align all rank $\frac 32$ roots on every rank 2 geodesics, and thus write this distribution as a union of 2-periodic rank 2 strips (in the third direction) of height 1. Clearly, the special distribution $D_0$ is distinct from any distribution of the first two types.

(b) In \cite{Pauli} we studied a family of Pauli puzzles, associated with a Moebius--Kantor complex (the ``Pauli complex'') and classified them. It is not difficult to prove that the Pauli complex is even, and correspondingly, the Pauli puzzles are even (i.e., for any triangle, the number of roots of rank 2 associated with a large triangle is even). It is shown in \cite{root} that there exists a unique even Moebius--Kantor complex, and these results also establish that there is an explicit classification of all even root distributions similar to the odd case.
\end{remark}

\section{Proof of Theorem \ref{T - main theorem}}

In view of Theorem \ref{t - odd root distributions}, the classification of odd puzzles reduces to the classification of puzzles having one of three admissible families of root distributions. 

We shall begin with the simplest case, in which the root distribution is a union of 1-periodic rank 2 strips of height 2 (case 1 in Theorem \ref{t - odd root distributions}).  For these puzzles, the center geodesic in each one of these strips is of rank $\frac 3 2$, and therefore it suffices to classify these strips. 

 We start with a vertex $x$ in the center of such a strip having as ring the first ring on the left, for $s=0$ (Fig.\ \ref{Fig 2}), positioned as follows:
 
 
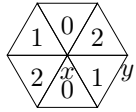
\begin{figure}[H]
 \begin{tikzpicture}[scale=.8]

\hexagon{1}{0}{2}{0}{1}{2}{0}{1}

  \draw node at (1,-.3) {$x$};
    
      \draw node at (2,-.3) {$y$};
    \end{tikzpicture}

\end{figure}

In this figure, the root of rank $\frac 3 2 $ at $x$ is horizontal. By Lemma \ref{L - edge labelling}, vertex to the right of this vertex is associated with $s=1$. We remark:

\begin{lemma}\label{L - unique extension}
There exists a unique position for a ring in $\Theta_1$ (for $s=1$) with the property that the root of rank $\frac 3 2 $ at $y$ is horizontal. 
\end{lemma}

\begin{proof}
We indicate the edge labeling with exponents as in \ts\ref{S - Meso}. We are required to find the word $^2 1^1 2^2$ and extend to the right positioning the root of rank $\frac 3 2$ horizontally.

\begin{figure}[H]
 \begin{tikzpicture}[scale=.8]
 
     \draw node at (1.7,-.1) {\tiny 1};
     \draw node at (1.8,.5) {\tiny 2};
     \draw node at (1.8,-.5) {\tiny 2};
     
\hexagon{1}{0}{2}{0}{1}{2}{0}{1}
\hexagon{2}{0}{2}{0}{}{}{2}{1}
\end{tikzpicture}
\caption{}\label{Fig - unique extension 32}
\end{figure}
\noindent Inspecting the three rings in $\Theta_1$, there is a unique solution as shown in the figure above.
\end{proof}

We should now extend Fig.\  \ref{Fig - unique extension 32} to the right. Note that the vertex ring is now an element of $\Theta_2$ corresponding to $s=2$. It is easily verified that the analog of Lemma \ref{L - unique extension} holds. By induction, this gives a period 6 strip of height 2 as shown in the introduction. At each step, it is easily verified that an analog of Lemma \ref{L - unique extension} holds for the given configuration. Note that this  strip makes use of three of the nine rings, and the three strips shown in the introduction exhaust all possible rings. This completes the proof of Theorem \ref{T - main theorem}, Case (2).

We  now consider Case (1). It follows by Theorem \ref{t - odd root distributions} that the roots distributions are 2-periodic. Here the strips have height 1, and there are two positions on both sides for the roots (which are non-horizontal). This gives four initial positions in total, which can be represented as follows:

\begin{figure}[H]
\begin{tikzpicture}[scale=.7]

\begin{scope}[shift={(0,0)}]
\coordinate (O) at (0.0,0.0);
\coordinate (v1) at (1.0,0.0);
\coordinate (v2) at (0.5,0.87);
\coordinate (v3) at (-0.5,0.87);
\coordinate (v4) at (-1.0,-0.0);
\coordinate (v5) at (-0.5,-0.87);
\coordinate (v6) at (0.5,-0.87);
\draw[solid,thin,color=black] (O) -- (v1) -- (v2) -- (O);

\draw[solid,very thick,rotate=60,shift=+(O)] (-0.2,0) -- (.2,0);
\draw[solid,very thick,rotate=120,shift=+(v1)] (-0.2,0) -- (.2,0);
\draw[solid,very thick,rotate=60,shift=+(v2)] (-0.2,0) -- (.2,0);

  \draw node at (1,-.5) {$x$};

\end{scope}
\begin{scope}[shift={(3,0)}]
\coordinate (O) at (0.0,0.0);
\coordinate (v1) at (1.0,0.0);
\coordinate (v2) at (0.5,0.87);
\coordinate (v3) at (-0.5,0.87);
\coordinate (v4) at (-1.0,-0.0);
\coordinate (v5) at (-0.5,-0.87);
\coordinate (v6) at (0.5,-0.87);
\draw[solid,thin,color=black] (O) -- (v1) -- (v2) -- (O);

\draw[solid,very thick,rotate=60,shift=+(O)] (-0.2,0) -- (.2,0);
\draw[solid,very thick,rotate=120,shift=+(v1)] (-0.2,0) -- (.2,0);
\draw[solid,very thick,rotate=120,shift=+(v2)] (-0.2,0) -- (.2,0);

  \draw node at (1,-.5) {$x$};

\end{scope}
\begin{scope}[shift={(6,0)}]
\coordinate (O) at (0.0,0.0);
\coordinate (v1) at (1.0,0.0);
\coordinate (v2) at (0.5,0.87);
\coordinate (v3) at (-0.5,0.87);
\coordinate (v4) at (-1.0,-0.0);
\coordinate (v5) at (-0.5,-0.87);
\coordinate (v6) at (0.5,-0.87);
\draw[solid,thin,color=black] (O) -- (v1) -- (v2) -- (O);

\draw[solid,very thick,rotate=120,shift=+(O)] (-0.2,0) -- (.2,0);
\draw[solid,very thick,rotate=60,shift=+(v1)] (-0.2,0) -- (.2,0);
\draw[solid,very thick,rotate=60,shift=+(v2)] (-0.2,0) -- (.2,0);

  \draw node at (1,-.5) {$x$};

\end{scope}
\begin{scope}[shift={(9,0)}]
\coordinate (O) at (0.0,0.0);
\coordinate (v1) at (1.0,0.0);
\coordinate (v2) at (0.5,0.87);
\coordinate (v3) at (-0.5,0.87);
\coordinate (v4) at (-1.0,-0.0);
\coordinate (v5) at (-0.5,-0.87);
\coordinate (v6) at (0.5,-0.87);
\draw[solid,thin,color=black] (O) -- (v1) -- (v2) -- (O);

\draw[solid,very thick,rotate=120,shift=+(O)] (-0.2,0) -- (.2,0);
\draw[solid,very thick,rotate=60,shift=+(v1)] (-0.2,0) -- (.2,0);
\draw[solid,very thick,rotate=120,shift=+(v2)] (-0.2,0) -- (.2,0);

  \draw node at (1,-.5) {$x$};

\end{scope}

\end{tikzpicture}
\caption{}\label{F - root distribution cases}
\end{figure}

Note that the root distribution must extend uniquely in the horizontal direction, starting from the initial configuration, in order to respect 2-periodicity. We assume as above that the ring at $x$ belongs to $\Theta_0$ with an horizontal edge labelled 0. For each case, we have, in principle, six possible initial labels (three on the triangle and two adjacent to it)  for a 1-strip, which gives 24 possibilities in total; however, each 1-strip contributes to several of these labelling---namely, four each, as we shall see now, for a total of six strips.   

Starting with the left hand side configuration in Fig. \ref{F - root distribution cases},  the following are the initial marking on a 1-strip:

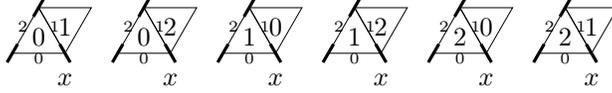
\begin{figure}[H]
\begin{tikzpicture}[scale=.7]

\begin{scope}[shift={(0,0)}]
\coordinate (O) at (0.0,0.0);
\coordinate (v1) at (1.0,0.0);
\coordinate (v2) at (0.5,0.87);

\draw[solid,very thick,rotate=60,shift=+(O)] (-0.2,0) -- (.2,0);
\draw[solid,very thick,rotate=120,shift=+(v1)] (-0.2,0) -- (.2,0);
\draw[solid,very thick,rotate=60,shift=+(v2)] (-0.2,0) -- (.2,0);

  \draw node at (1,-.5) {$x$};
     \draw node at (.5,-.1) {\tiny 0};
     \draw node at (.8,.5) {\tiny 1};
     \draw node at (0.2,.5) {\tiny 2};
\up{0}{0}{0}
\down{1}{0}{1}
\end{scope}

\begin{scope}[shift={(2,0)}]
\coordinate (O) at (0.0,0.0);
\coordinate (v1) at (1.0,0.0);
\coordinate (v2) at (0.5,0.87);
\coordinate (v3) at (-0.5,0.87);
\coordinate (v4) at (-1.0,-0.0);
\coordinate (v5) at (-0.5,-0.87);
\coordinate (v6) at (0.5,-0.87);

\draw[solid,very thick,rotate=60,shift=+(O)] (-0.2,0) -- (.2,0);
\draw[solid,very thick,rotate=120,shift=+(v1)] (-0.2,0) -- (.2,0);
\draw[solid,very thick,rotate=60,shift=+(v2)] (-0.2,0) -- (.2,0);

  \draw node at (1,-.5) {$x$};
     \draw node at (.5,-.1) {\tiny 0};
     \draw node at (.8,.5) {\tiny 1};
     \draw node at (0.2,.5) {\tiny 2};
\up{0}{0}{0}
\down{1}{0}{2}
\end{scope}

\begin{scope}[shift={(4,0)}]
\coordinate (O) at (0.0,0.0);
\coordinate (v1) at (1.0,0.0);
\coordinate (v2) at (0.5,0.87);
\coordinate (v3) at (-0.5,0.87);
\coordinate (v4) at (-1.0,-0.0);
\coordinate (v5) at (-0.5,-0.87);
\coordinate (v6) at (0.5,-0.87);

\draw[solid,very thick,rotate=60,shift=+(O)] (-0.2,0) -- (.2,0);
\draw[solid,very thick,rotate=120,shift=+(v1)] (-0.2,0) -- (.2,0);
\draw[solid,very thick,rotate=60,shift=+(v2)] (-0.2,0) -- (.2,0);

  \draw node at (1,-.5) {$x$};
     \draw node at (.5,-.1) {\tiny 0};
     \draw node at (.8,.5) {\tiny 1};
     \draw node at (0.2,.5) {\tiny 2};
\up{0}{0}{1}
\down{1}{0}{0}
\end{scope}

\begin{scope}[shift={(6,0)}]
\coordinate (O) at (0.0,0.0);
\coordinate (v1) at (1.0,0.0);
\coordinate (v2) at (0.5,0.87);
\coordinate (v3) at (-0.5,0.87);
\coordinate (v4) at (-1.0,-0.0);
\coordinate (v5) at (-0.5,-0.87);
\coordinate (v6) at (0.5,-0.87);

\draw[solid,very thick,rotate=60,shift=+(O)] (-0.2,0) -- (.2,0);
\draw[solid,very thick,rotate=120,shift=+(v1)] (-0.2,0) -- (.2,0);
\draw[solid,very thick,rotate=60,shift=+(v2)] (-0.2,0) -- (.2,0);

  \draw node at (1,-.5) {$x$};
     \draw node at (.5,-.1) {\tiny 0};
     \draw node at (.8,.5) {\tiny 1};
     \draw node at (0.2,.5) {\tiny 2};
\up{0}{0}{1}
\down{1}{0}{2}
\end{scope}

\begin{scope}[shift={(8,0)}]
\coordinate (O) at (0.0,0.0);
\coordinate (v1) at (1.0,0.0);
\coordinate (v2) at (0.5,0.87);
\coordinate (v3) at (-0.5,0.87);
\coordinate (v4) at (-1.0,-0.0);
\coordinate (v5) at (-0.5,-0.87);
\coordinate (v6) at (0.5,-0.87);

\draw[solid,very thick,rotate=60,shift=+(O)] (-0.2,0) -- (.2,0);
\draw[solid,very thick,rotate=120,shift=+(v1)] (-0.2,0) -- (.2,0);
\draw[solid,very thick,rotate=60,shift=+(v2)] (-0.2,0) -- (.2,0);

  \draw node at (1,-.5) {$x$};
     \draw node at (.5,-.1) {\tiny 0};
     \draw node at (.8,.5) {\tiny 1};
     \draw node at (0.2,.5) {\tiny 2};
\up{0}{0}{2}
\down{1}{0}{0}
\end{scope}

\begin{scope}[shift={(10,0)}]
\coordinate (O) at (0.0,0.0);
\coordinate (v1) at (1.0,0.0);
\coordinate (v2) at (0.5,0.87);
\coordinate (v3) at (-0.5,0.87);
\coordinate (v4) at (-1.0,-0.0);
\coordinate (v5) at (-0.5,-0.87);
\coordinate (v6) at (0.5,-0.87);

\draw[solid,very thick,rotate=60,shift=+(O)] (-0.2,0) -- (.2,0);
\draw[solid,very thick,rotate=120,shift=+(v1)] (-0.2,0) -- (.2,0);
\draw[solid,very thick,rotate=60,shift=+(v2)] (-0.2,0) -- (.2,0);

  \draw node at (1,-.5) {$x$};
     \draw node at (.5,-.1) {\tiny 0};
     \draw node at (.8,.5) {\tiny 1};
     \draw node at (0.2,.5) {\tiny 2};
\up{0}{0}{2}
\down{1}{0}{1}
\end{scope}
\end{tikzpicture}
\end{figure}
 
 The argument in Lemma \ref{L - unique extension} shows again that each of these strip extend uniquely to the right. In the figure below, we have labelled with an $x$ all vertices in either side which have a ring in $\Theta_0$ (for $s=0$).

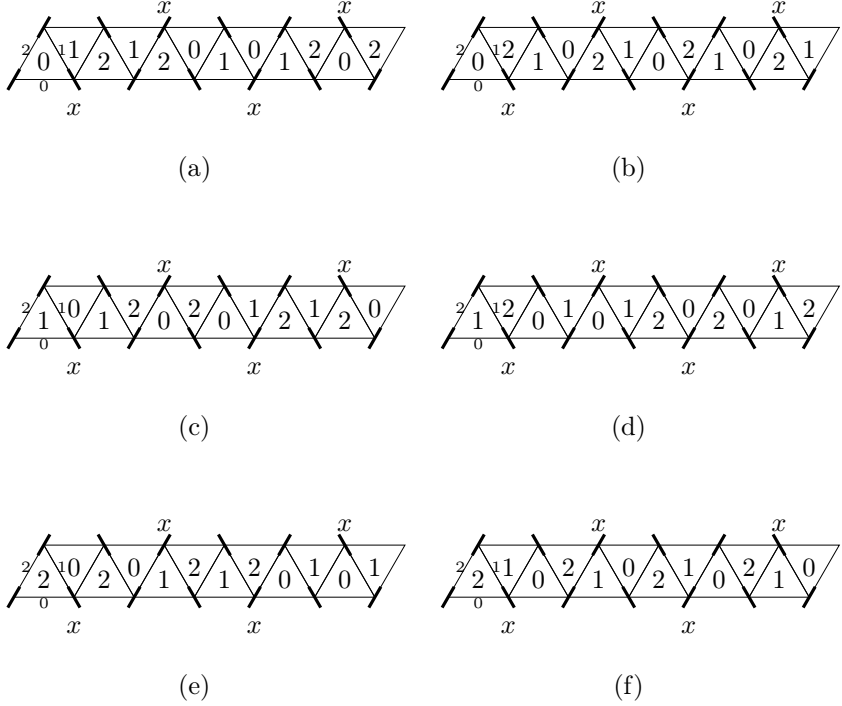
\begin{figure}[H]
 \begin{tikzpicture}[scale=.8]
 
     \draw node at (.5,-.1) {\tiny 0};
     \draw node at (.8,.5) {\tiny 1};
     \draw node at (0.2,.5) {\tiny 2};
     
\up{0}{0}{0}
\down{1}{0}{1}
\up{1}{0}{2}
\down{2}{0}{1}
\up{2}{0}{2}
\down{3}{0}{0}
\up{3}{0}{1}
\down{4}{0}{0}
\up{4}{0}{1}
\down{5}{0}{2}
\up{5}{0}{0}
\down{6}{0}{2}

  \draw node at (1,-.5) {$x$};
  \draw node at (4,-.5) {$x$};
  \draw node at (2.5,1.2) {$x$};
  \draw node at (5.5,1.2) {$x$};
\coordinate (O) at (0.0,0.0);
\coordinate (v1) at (1.0,0.0);
\coordinate (v2) at (0.5,0.87);
\coordinate (v3) at (2.0,0.0);
\coordinate (v4) at (3.0,0.0);
\coordinate (v5) at (4.0,0.0);
\coordinate (v6) at (5.0,0.0);
\coordinate (v7) at (6.0,0.0);
\coordinate (v8) at (1.5,0.87);
\coordinate (v9) at (2.5,0.87);
\coordinate (v10) at (3.5,0.87);
\coordinate (v11) at (4.5,0.87);
\coordinate (v12) at (5.5,0.87);

\draw[solid,very thick,rotate=60,shift=+(O)] (-0.2,0) -- (.2,0);
\draw[solid,very thick,rotate=120,shift=+(v1)] (-0.2,0) -- (.2,0);
\draw[solid,very thick,rotate=60,shift=+(v3)] (-0.2,0) -- (.2,0);
\draw[solid,very thick,rotate=120,shift=+(v4)] (-0.2,0) -- (.2,0);
\draw[solid,very thick,rotate=60,shift=+(v5)] (-0.2,0) -- (.2,0);
\draw[solid,very thick,rotate=120,shift=+(v6)] (-0.2,0) -- (.2,0);
\draw[solid,very thick,rotate=60,shift=+(v7)] (-0.2,0) -- (.2,0);
\draw[solid,very thick,rotate=60,shift=+(v2)] (-0.2,0) -- (.2,0);
\draw[solid,very thick,rotate=120,shift=+(v8)] (-0.2,0) -- (.2,0);
\draw[solid,very thick,rotate=60,shift=+(v9)] (-0.2,0) -- (.2,0);
\draw[solid,very thick,rotate=120,shift=+(v10)] (-0.2,0) -- (.2,0);
\draw[solid,very thick,rotate=60,shift=+(v11)] (-0.2,0) -- (.2,0);
\draw[solid,very thick,rotate=120,shift=+(v12)] (-0.2,0) -- (.2,0);

  \draw node at (3,-1.5) {(a)};

    \end{tikzpicture}\ \ \ \ 
 \begin{tikzpicture}[scale=.8]
 
     \draw node at (.5,-.1) {\tiny 0};
     \draw node at (.8,.5) {\tiny 1};
     \draw node at (0.2,.5) {\tiny 2};
     
\up{0}{0}{0}
\down{1}{0}{2}
\up{1}{0}{1}
\down{2}{0}{0}
\up{2}{0}{2}
\down{3}{0}{1}
\up{3}{0}{0}
\down{4}{0}{2}
\up{4}{0}{1}
\down{5}{0}{0}
\up{5}{0}{2}
\down{6}{0}{1}

  \draw node at (1,-.5) {$x$};
  \draw node at (4,-.5) {$x$};
  \draw node at (2.5,1.2) {$x$};
  \draw node at (5.5,1.2) {$x$};
\coordinate (O) at (0.0,0.0);
\coordinate (v1) at (1.0,0.0);
\coordinate (v2) at (0.5,0.87);
\coordinate (v3) at (2.0,0.0);
\coordinate (v4) at (3.0,0.0);
\coordinate (v5) at (4.0,0.0);
\coordinate (v6) at (5.0,0.0);
\coordinate (v7) at (6.0,0.0);
\coordinate (v8) at (1.5,0.87);
\coordinate (v9) at (2.5,0.87);
\coordinate (v10) at (3.5,0.87);
\coordinate (v11) at (4.5,0.87);
\coordinate (v12) at (5.5,0.87);

\draw[solid,very thick,rotate=60,shift=+(O)] (-0.2,0) -- (.2,0);
\draw[solid,very thick,rotate=120,shift=+(v1)] (-0.2,0) -- (.2,0);
\draw[solid,very thick,rotate=60,shift=+(v3)] (-0.2,0) -- (.2,0);
\draw[solid,very thick,rotate=120,shift=+(v4)] (-0.2,0) -- (.2,0);
\draw[solid,very thick,rotate=60,shift=+(v5)] (-0.2,0) -- (.2,0);
\draw[solid,very thick,rotate=120,shift=+(v6)] (-0.2,0) -- (.2,0);
\draw[solid,very thick,rotate=60,shift=+(v7)] (-0.2,0) -- (.2,0);
\draw[solid,very thick,rotate=60,shift=+(v2)] (-0.2,0) -- (.2,0);
\draw[solid,very thick,rotate=120,shift=+(v8)] (-0.2,0) -- (.2,0);
\draw[solid,very thick,rotate=60,shift=+(v9)] (-0.2,0) -- (.2,0);
\draw[solid,very thick,rotate=120,shift=+(v10)] (-0.2,0) -- (.2,0);
\draw[solid,very thick,rotate=60,shift=+(v11)] (-0.2,0) -- (.2,0);
\draw[solid,very thick,rotate=120,shift=+(v12)] (-0.2,0) -- (.2,0);

  \draw node at (3,-1.5) {(b)};

    \end{tikzpicture}
    \end{figure}

\begin{figure}[H]
 \begin{tikzpicture}[scale=.8]
 
     \draw node at (.5,-.1) {\tiny 0};
     \draw node at (.8,.5) {\tiny 1};
     \draw node at (0.2,.5) {\tiny 2};
     
\up{0}{0}{1}
\down{1}{0}{0}
\up{1}{0}{1}
\down{2}{0}{2}
\up{2}{0}{0}
\down{3}{0}{2}
\up{3}{0}{0}
\down{4}{0}{1}
\up{4}{0}{2}
\down{5}{0}{1}
\up{5}{0}{2}
\down{6}{0}{0}
  \draw node at (1,-.5) {$x$};
  \draw node at (4,-.5) {$x$};
  \draw node at (2.5,1.2) {$x$};
  \draw node at (5.5,1.2) {$x$};
\coordinate (O) at (0.0,0.0);
\coordinate (v1) at (1.0,0.0);
\coordinate (v2) at (0.5,0.87);
\coordinate (v3) at (2.0,0.0);
\coordinate (v4) at (3.0,0.0);
\coordinate (v5) at (4.0,0.0);
\coordinate (v6) at (5.0,0.0);
\coordinate (v7) at (6.0,0.0);
\coordinate (v8) at (1.5,0.87);
\coordinate (v9) at (2.5,0.87);
\coordinate (v10) at (3.5,0.87);
\coordinate (v11) at (4.5,0.87);
\coordinate (v12) at (5.5,0.87);

\draw[solid,very thick,rotate=60,shift=+(O)] (-0.2,0) -- (.2,0);
\draw[solid,very thick,rotate=120,shift=+(v1)] (-0.2,0) -- (.2,0);
\draw[solid,very thick,rotate=60,shift=+(v3)] (-0.2,0) -- (.2,0);
\draw[solid,very thick,rotate=120,shift=+(v4)] (-0.2,0) -- (.2,0);
\draw[solid,very thick,rotate=60,shift=+(v5)] (-0.2,0) -- (.2,0);
\draw[solid,very thick,rotate=120,shift=+(v6)] (-0.2,0) -- (.2,0);
\draw[solid,very thick,rotate=60,shift=+(v7)] (-0.2,0) -- (.2,0);
\draw[solid,very thick,rotate=60,shift=+(v2)] (-0.2,0) -- (.2,0);
\draw[solid,very thick,rotate=120,shift=+(v8)] (-0.2,0) -- (.2,0);
\draw[solid,very thick,rotate=60,shift=+(v9)] (-0.2,0) -- (.2,0);
\draw[solid,very thick,rotate=120,shift=+(v10)] (-0.2,0) -- (.2,0);
\draw[solid,very thick,rotate=60,shift=+(v11)] (-0.2,0) -- (.2,0);
\draw[solid,very thick,rotate=120,shift=+(v12)] (-0.2,0) -- (.2,0);

  \draw node at (3,-1.5) {(c)};

    \end{tikzpicture}
\ \ \ 
 \begin{tikzpicture}[scale=.8]
 
     \draw node at (.5,-.1) {\tiny 0};
     \draw node at (.8,.5) {\tiny 1};
     \draw node at (0.2,.5) {\tiny 2};
     
\up{0}{0}{1}
\down{1}{0}{2}
\up{1}{0}{0}
\down{2}{0}{1}
\up{2}{0}{0}
\down{3}{0}{1}
\up{3}{0}{2}
\down{4}{0}{0}
\up{4}{0}{2}
\down{5}{0}{0}
\up{5}{0}{1}
\down{6}{0}{2}

  \draw node at (1,-.5) {$x$};
  \draw node at (4,-.5) {$x$};
  \draw node at (2.5,1.2) {$x$};
  \draw node at (5.5,1.2) {$x$};
\coordinate (O) at (0.0,0.0);
\coordinate (v1) at (1.0,0.0);
\coordinate (v2) at (0.5,0.87);
\coordinate (v3) at (2.0,0.0);
\coordinate (v4) at (3.0,0.0);
\coordinate (v5) at (4.0,0.0);
\coordinate (v6) at (5.0,0.0);
\coordinate (v7) at (6.0,0.0);
\coordinate (v8) at (1.5,0.87);
\coordinate (v9) at (2.5,0.87);
\coordinate (v10) at (3.5,0.87);
\coordinate (v11) at (4.5,0.87);
\coordinate (v12) at (5.5,0.87);

\draw[solid,very thick,rotate=60,shift=+(O)] (-0.2,0) -- (.2,0);
\draw[solid,very thick,rotate=120,shift=+(v1)] (-0.2,0) -- (.2,0);
\draw[solid,very thick,rotate=60,shift=+(v3)] (-0.2,0) -- (.2,0);
\draw[solid,very thick,rotate=120,shift=+(v4)] (-0.2,0) -- (.2,0);
\draw[solid,very thick,rotate=60,shift=+(v5)] (-0.2,0) -- (.2,0);
\draw[solid,very thick,rotate=120,shift=+(v6)] (-0.2,0) -- (.2,0);
\draw[solid,very thick,rotate=60,shift=+(v7)] (-0.2,0) -- (.2,0);
\draw[solid,very thick,rotate=60,shift=+(v2)] (-0.2,0) -- (.2,0);
\draw[solid,very thick,rotate=120,shift=+(v8)] (-0.2,0) -- (.2,0);
\draw[solid,very thick,rotate=60,shift=+(v9)] (-0.2,0) -- (.2,0);
\draw[solid,very thick,rotate=120,shift=+(v10)] (-0.2,0) -- (.2,0);
\draw[solid,very thick,rotate=60,shift=+(v11)] (-0.2,0) -- (.2,0);
\draw[solid,very thick,rotate=120,shift=+(v12)] (-0.2,0) -- (.2,0);

  \draw node at (3,-1.5) {(d)};

    \end{tikzpicture}

    \end{figure}

\begin{figure}[H]
 \begin{tikzpicture}[scale=.8]
 
     \draw node at (.5,-.1) {\tiny 0};
     \draw node at (.8,.5) {\tiny 1};
     \draw node at (0.2,.5) {\tiny 2};
     
\up{0}{0}{2}
\down{1}{0}{0}
\up{1}{0}{2}
\down{2}{0}{0}
\up{2}{0}{1}
\down{3}{0}{2}
\up{3}{0}{1}
\down{4}{0}{2}
\up{4}{0}{0}
\down{5}{0}{1}
\up{5}{0}{0}
\down{6}{0}{1}
  \draw node at (1,-.5) {$x$};
  \draw node at (4,-.5) {$x$};
  \draw node at (2.5,1.2) {$x$};
  \draw node at (5.5,1.2) {$x$};
\coordinate (O) at (0.0,0.0);
\coordinate (v1) at (1.0,0.0);
\coordinate (v2) at (0.5,0.87);
\coordinate (v3) at (2.0,0.0);
\coordinate (v4) at (3.0,0.0);
\coordinate (v5) at (4.0,0.0);
\coordinate (v6) at (5.0,0.0);
\coordinate (v7) at (6.0,0.0);
\coordinate (v8) at (1.5,0.87);
\coordinate (v9) at (2.5,0.87);
\coordinate (v10) at (3.5,0.87);
\coordinate (v11) at (4.5,0.87);
\coordinate (v12) at (5.5,0.87);

\draw[solid,very thick,rotate=60,shift=+(O)] (-0.2,0) -- (.2,0);
\draw[solid,very thick,rotate=120,shift=+(v1)] (-0.2,0) -- (.2,0);
\draw[solid,very thick,rotate=60,shift=+(v3)] (-0.2,0) -- (.2,0);
\draw[solid,very thick,rotate=120,shift=+(v4)] (-0.2,0) -- (.2,0);
\draw[solid,very thick,rotate=60,shift=+(v5)] (-0.2,0) -- (.2,0);
\draw[solid,very thick,rotate=120,shift=+(v6)] (-0.2,0) -- (.2,0);
\draw[solid,very thick,rotate=60,shift=+(v7)] (-0.2,0) -- (.2,0);
\draw[solid,very thick,rotate=60,shift=+(v2)] (-0.2,0) -- (.2,0);
\draw[solid,very thick,rotate=120,shift=+(v8)] (-0.2,0) -- (.2,0);
\draw[solid,very thick,rotate=60,shift=+(v9)] (-0.2,0) -- (.2,0);
\draw[solid,very thick,rotate=120,shift=+(v10)] (-0.2,0) -- (.2,0);
\draw[solid,very thick,rotate=60,shift=+(v11)] (-0.2,0) -- (.2,0);
\draw[solid,very thick,rotate=120,shift=+(v12)] (-0.2,0) -- (.2,0);

  \draw node at (3,-1.5) {(e)};

    \end{tikzpicture}
\ \ \ 
 \begin{tikzpicture}[scale=.8]
 
     \draw node at (.5,-.1) {\tiny 0};
     \draw node at (.8,.5) {\tiny 1};
     \draw node at (0.2,.5) {\tiny 2};
     
\up{0}{0}{2}
\down{1}{0}{1}
\up{1}{0}{0}
\down{2}{0}{2}
\up{2}{0}{1}
\down{3}{0}{0}
\up{3}{0}{2}
\down{4}{0}{1}
\up{4}{0}{0}
\down{5}{0}{2}
\up{5}{0}{1}
\down{6}{0}{0}

  \draw node at (1,-.5) {$x$};
  \draw node at (4,-.5) {$x$};
  \draw node at (2.5,1.2) {$x$};
  \draw node at (5.5,1.2) {$x$};
\coordinate (O) at (0.0,0.0);
\coordinate (v1) at (1.0,0.0);
\coordinate (v2) at (0.5,0.87);
\coordinate (v3) at (2.0,0.0);
\coordinate (v4) at (3.0,0.0);
\coordinate (v5) at (4.0,0.0);
\coordinate (v6) at (5.0,0.0);
\coordinate (v7) at (6.0,0.0);
\coordinate (v8) at (1.5,0.87);
\coordinate (v9) at (2.5,0.87);
\coordinate (v10) at (3.5,0.87);
\coordinate (v11) at (4.5,0.87);
\coordinate (v12) at (5.5,0.87);

\draw[solid,very thick,rotate=60,shift=+(O)] (-0.2,0) -- (.2,0);
\draw[solid,very thick,rotate=120,shift=+(v1)] (-0.2,0) -- (.2,0);
\draw[solid,very thick,rotate=60,shift=+(v3)] (-0.2,0) -- (.2,0);
\draw[solid,very thick,rotate=120,shift=+(v4)] (-0.2,0) -- (.2,0);
\draw[solid,very thick,rotate=60,shift=+(v5)] (-0.2,0) -- (.2,0);
\draw[solid,very thick,rotate=120,shift=+(v6)] (-0.2,0) -- (.2,0);
\draw[solid,very thick,rotate=60,shift=+(v7)] (-0.2,0) -- (.2,0);
\draw[solid,very thick,rotate=60,shift=+(v2)] (-0.2,0) -- (.2,0);
\draw[solid,very thick,rotate=120,shift=+(v8)] (-0.2,0) -- (.2,0);
\draw[solid,very thick,rotate=60,shift=+(v9)] (-0.2,0) -- (.2,0);
\draw[solid,very thick,rotate=120,shift=+(v10)] (-0.2,0) -- (.2,0);
\draw[solid,very thick,rotate=60,shift=+(v11)] (-0.2,0) -- (.2,0);
\draw[solid,very thick,rotate=120,shift=+(v12)] (-0.2,0) -- (.2,0);

  \draw node at (3,-1.5) {(f)};

    \end{tikzpicture}

    \end{figure}

These 6 strips exhaust all possible (i.e., 24) configurations at $x$ (taking into account both the markings and the distributions). Furthermore, it is clear that the two strips (a) and (d) are identical as 1-strips (with opposite root distributions) and similarly for strips (c) and (e); strip (b) is a glided symmetry of itself (with opposite root distribution), and similarly for strip (f). This gives a total of 4 distinct strips of period 6 in this category of puzzles, as stated in Th.\ \ref{T - main theorem}, which also describes how these strips can be attached to each other into an odd puzzle.

Finally, it remains to describe the puzzles corresponding to the special root distribution $D_0$. The statement of Th.\  \ref{T - main theorem} lists all possible markings respecting this distribution, starting from the centerpiece. Since $D_0$ is unique, the centerpiece can be positioned in precisely two ways relative to the edge marking, namely:

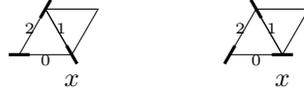
\begin{figure}[H]
\begin{tikzpicture}[scale=.7]

\begin{scope}[shift={(0,0)}]
\coordinate (O) at (0.0,0.0);
\coordinate (v1) at (1.0,0.0);
\coordinate (v2) at (0.5,0.87);

\draw[solid,very thick,rotate=0,shift=+(O)] (-0.2,0) -- (.2,0);
\draw[solid,very thick,rotate=120,shift=+(v1)] (-0.2,0) -- (.2,0);
\draw[solid,very thick,rotate=60,shift=+(v2)] (-0.2,0) -- (.2,0);

  \draw node at (1,-.5) {$x$};
     \draw node at (.5,-.1) {\tiny 0};
     \draw node at (.8,.5) {\tiny 1};
     \draw node at (0.2,.5) {\tiny 2};
\up{0}{0}{}
\down{1}{0}{}
\end{scope}

\begin{scope}[shift={(4,0)}]
\coordinate (O) at (0.0,0.0);
\coordinate (v1) at (1.0,0.0);
\coordinate (v2) at (0.5,0.87);

\draw[solid,very thick,rotate=60,shift=+(O)] (-0.2,0) -- (.2,0);
\draw[solid,very thick,rotate=0,shift=+(v1)] (-0.2,0) -- (.2,0);
\draw[solid,very thick,rotate=120,shift=+(v2)] (-0.2,0) -- (.2,0);

  \draw node at (1,-.5) {$x$};
     \draw node at (.5,-.1) {\tiny 0};
     \draw node at (.8,.5) {\tiny 1};
     \draw node at (0.2,.5) {\tiny 2};
\up{0}{0}{}
\down{1}{0}{}
\end{scope}
\end{tikzpicture}
\end{figure}

In both cases, there are three possible markings on the center triangle, and two possible markings for the adjacent triangle. By an analog of Lemma \ref{L - unique extension}, the extensions are then unique once these two triangles are chosen, which explains why there are twelve possible puzzles in total. Theorem \ref{T - main theorem} provides the computation of these markings explicit on a ball of radius 1 around the center triangle.

     This concludes the proof of our theorem.

\end{document}